\documentclass[times,final]{elsarticle}
\usepackage{jcomp}
\usepackage{framed,multirow}
\usepackage{amssymb}
\usepackage{latexsym}
\usepackage{amsmath,amstext,amsthm,amssymb,bbm,amsbsy}
\usepackage{mathtools}
\usepackage{tabularx}

\newcommand{\be}{\begin{equation}}
\newcommand{\ee}{\end{equation}}
\newcommand{\benn}{\begin{equation*}}
\newcommand{\eenn}{\end{equation*}}

\newtheorem{lem}{Lemma}[section]
\newtheorem{rem}{Remark}[section]

\newtheorem{thm}{Theorem}[section]
\newtheorem{defn}{Definition}

\newtheorem{example}{Example}[section]

\def\bx {\boldsymbol{x}}
\def\by {\boldsymbol{y}}
\def\bv {\boldsymbol{v}}
\def\bu {\boldsymbol{u}}
\def\bd {\boldsymbol{d}}
\def\bp {\boldsymbol{p}}
\def\R {\mathbb{R}}

\def\dom {\mathrm{dom~}}

\def\co {\mathrm{co~}}

\def\unitsim {\Lambda}
\def\Rn {\R^n}
\def\gmRn {\Gamma_0(\R^n)}

\def\HHJ {H}
\def\JHJ {J}
\def\LNN {L}
\def\JNN {\tilde{J}}
\def\SLO {S_{LO}}

\usepackage{url}
\usepackage{xcolor}
\definecolor{newcolor}{rgb}{.8,.349,.1}

\journal{Journal of Computational Physics}

\begin{document}

\verso{J\'er\^ome Darbon, Tingwei Meng}

\begin{frontmatter}

\title{On some neural network architectures that can represent viscosity solutions of certain high dimensional Hamilton--Jacobi partial differential equations\tnoteref{tnote1}}%
\tnotetext[tnote1]{
Authors' names are given in last/family name alphabetical order.}

\author[1]{J\'er\^ome \snm{Darbon}}
\ead{jerome\_darbon@brown.edu}
\author[1]{Tingwei \snm{Meng}}
\ead{tingwei\_meng@brown.edu}

\address[1]{Department of Applied Mathematics, Brown University, Providence, RI, 02912, USA}

\received{}
\finalform{}
\accepted{}
\availableonline{}
\communicated{}

\begin{abstract}
    We propose novel connections between several neural network architectures and viscosity solutions of some Hamilton--Jacobi (HJ) partial differential equations (PDEs) whose Hamiltonian is convex and only depends on the spatial gradient of the solution. To be specific, we prove that under certain assumptions, the two neural network architectures we proposed represent viscosity solutions to two sets of HJ PDEs with zero error. 
    We also implement our proposed neural network architectures using Tensorflow and provide several examples and illustrations. Note that these neural network representations can avoid curve of dimensionality for certain HJ PDEs, since they do not involve neither grids nor discretization. 
    Our results suggest that efficient dedicated hardware implementation for neural networks can be leveraged to evaluate viscosity solutions of certain HJ PDEs.
\end{abstract}


\end{frontmatter}

\section{Introduction}

Hamilton--Jacobi (HJ) partial differential equations (PDEs) arise in areas such as physics \cite{Arnold1989Math, Caratheodory1965CalculusI,Caratheodory1967CalculusII,Courant1989Methods,landau1978course},
optimal control \cite{Bardi1997Optimal, Elliott1987Viscosity,fleming1976deterministic,fleming2006controlled,mceneaney2006max}, game theory \cite{BARRON1984213, Buckdahn2011Recent, Evans1984Differential, Ishii1988Representation}, and imaging sciences \cite{darbon2015convex,darbon2019decomposition,Darbon2016Algorithms}. 
In this paper, we consider HJ PDEs with state and time independent  Hamiltonian function $H\colon\R^n\to \R$
and initial data 
$J\colon\R^n \to \R$ that read as follows
\begin{equation} \label{eqn:intro_example}
\begin{dcases} 
\frac{\partial S}{\partial t}(\bx,t)+H(\nabla_{\bx}S(\bx,t))= 0 & \mbox{{\rm in} }\mathbb{R}^{n}\times(0,+\infty),\\
S(\bx,0)=J(\bx) & \mbox{{\rm in} }\mathbb{R}^{n}.
\end{dcases}
\end{equation}
The partial derivative with respect to $t$ and the gradient vector with respect to $\bx$ of the solution $(\bx,t) \mapsto S(\bx,t)$  are denoted by 
 $\frac{\partial S}{\partial t}(\bx,t)$ and $\nabla_{\bx}S(\bx,t) =\left(\frac{\partial S}{\partial x_1} (\bx, t), \dots, \frac{\partial S}{\partial x_n} (\bx, t)\right)$, respectively.
 Note that the Hamiltonian $H$ only depends on $\nabla_{\bx} S(\bx,t)$.

Recently, \cite{darbon2019overcoming} establishes novel connections between some neural network architectures and  the viscosity solution of a set of HJ PDEs in the form of \eqref{eqn:intro_example}. (We refer readers to \cite{Bardi1997Optimal, bardi1984hopf, barles1994solutions, crandall1992user}
for the definition of the viscosity solution.) In \cite{darbon2019overcoming}, the authors provided the conditions under which their proposed neural network architecture represents the viscosity solution to the corresponding HJ PDEs whose initial data $J$ and Hamiltonian $H$ are related to the parameters in the neural network. Note that in the HJ PDEs they considered, the initial data $J$ is assumed to be a convex piecewise affine function, and the Hamiltonian $H$ also satisfies certain assumptions.

In this paper, we consider the HJ PDEs in the form of \eqref{eqn:intro_example} satisfying other assumptions. For instance, the Hamiltonian $H$ is convex, while the initial data $J$ is not necessarily convex. Under these assumptions, we prove that the two neural network architectures depicted in Figs.~\ref{fig: nn1} and \ref{fig: nn2} represent viscosity solutions to the corresponding HJ PDEs in the form of \eqref{eqn:intro_example} with initial data $J$ and convex Hamiltonian $H$. 
To be specific, in the first architecture shown in Fig.~\ref{fig: nn1}, the convex activation function $L$ in the neural network gives the Lagrangian function, whose Fenchel--Legendre transform gives the Hamiltonian $H$ in the corresponding HJ PDE. The initial data equals the minimum of several functions which are shifted copies of the asymptotic function $L'_\infty$ of $L$. The main result of this connection between the neural network architecture depicted in Fig.~\ref{fig: nn1} and the corresponding HJ PDE is stated in Thm.~\ref{thm:1}.
In the second architecture shown in Fig.~\ref{fig: nn2}, the activation function gives the initial data $J$ in the HJ PDE. The Hamiltonian $H$ is a piecewise affine convex function determined by the parameters in the neural network. The main result of this connection between the neural network architecture depicted in Fig.~\ref{fig: nn2} and the corresponding HJ PDE is stated in Thm.~\ref{thm:2}.

To summarize, this paper investigates the connection between several neural network architectures and some specific sets of HJ PDEs. The motivations and advantages of this work are listed as follows
\begin{itemize}
    \item Compared with traditional grid based representations, our proposed neural network representations do not involve any discretization of space and time. Hence these neural network representations can avoid the curse of dimensionality for certain HJ PDEs if the correct parameters are provided. 
    
    \item Our novel connections between certain HJ PDEs and neural networks suggest a possible direction to solve some HJ PDEs by leveraging efficient hardware technologies and silicon-based electric circuits dedicated to neural networks. 
    LeCun mentioned in \cite{lecun2019isscc} that the use of neural networks has been greatly influenced by available hardware. There have been many initiatives designing and constructing new hardware for extremely efficient (in terms of speed, latency, throughput or energy) implementations of neural networks.
    For instance, efficient neural network implementations are developed and optimized 
    using field programmable gate arrays \cite{farabet-suml-11,farabet-fpl-09,farabet.09.iccvw}, Intel's architecture \cite{banerjeeEtal2019sfi}, Google's ``Tensor Processor Unit" \cite{googleTPU17}, and certain building blocks \cite{kundu2019ktanh}. To obtain better performance on neural network computation, Xilinx announced a new set of hardware called Versal AI core, while Intel enhances their processors with specific hardware instructions.  
    In addition, there is an evolution of silicon-based electrical circuits for machine learning, for which we refer readers to \cite{chen2020classification, Hirjibehedin.20.nature}.
    LeCun also suggests in \cite[Sec.~3]{lecun2019isscc} possible new trends for hardware dedicated to neural networks. These trends for efficient neural network implementations motivate our study of the connections between neural network architectures and HJ PDEs.
    \item This work provides a possible interpretation of specific neural networks from the aspect of HJ PDEs.
\end{itemize}

\bigbreak
\noindent
\textbf{Literature review.} There is a huge body of literature on overcoming the curse of dimensionality of certain HJ PDEs.
These works include, but are not limited to, max-plus algebra methods \cite{mceneaney2006max,akian2006max,akian2008max, dower2015max,Fleming2000Max,gaubert2011curse,McEneaney2007COD,mceneaney2008curse,mceneaney2009convergence}, dynamic programming and reinforcement learning \cite{alla2019efficient,bertsekas2019reinforcement}, tensor decomposition techniques \cite{dolgov2019tensor,horowitz2014linear,todorov2009efficient}, sparse grids \cite{bokanowski2013adaptive,garcke2017suboptimal,kang2017mitigating}, model order reduction \cite{alla2017error,kunisch2004hjb}, polynomial approximation \cite{kalise2019robust,kalise2018polynomial}, optimization methods \cite{darbon2015convex,darbon2019decomposition,Darbon2016Algorithms,yegorov2017perspectives} and neural networks \cite{darbon2019overcoming,bachouch2018deep, Djeridane2006Neural,jiang2016using, Han2018Solving, hure2018deep, hure2019some, lambrianides2019new, Niarchos2006Neural, reisinger2019rectified,royo2016recursive, Sirignano2018DGM}. 

Recently, because of the trends for the efficient hardware implementations, neural networks have been increasingly applied in solving PDEs \cite{bachouch2018deep,Djeridane2006Neural,Han2018Solving, hure2018deep, hure2019some,lambrianides2019new,Niarchos2006Neural,reisinger2019rectified, royo2016recursive,Sirignano2018DGM,beck2018solving, beck2019deep,beck2019machine, Berg2018Unified,chan2019machine, Cheng2006Fixed, Dissanayake1994Neural,  dockhorn2019discussion, E2017Deep, Farimani2017Deep, Fujii2019Asymptotic, grohs2019deep, han2019solving, hsieh2018learning, jianyu2003numerical, khoo2017solving,khoo2019solving, Lagaris1998ANN, Lagaris2000NN,  lee1990neural, lye2019deep, McFall2009ANN, Meade1994Numerical, Milligen1995NN,  pham2019neural,  Rudd2014Constrained, Tang2017Study, Tassa2007Least, weinan2018deep, Yadav2015Intro, yang2018physics, yang2019adversarial,Zhao2017High,Kong2015Probabilistic} and inverse problems involving PDEs \cite{yang2018physics,long2017pde,long2019pde, meng2019composite, meng2019ppinn, pang2019fpinns, raissi2018deep,raissi2018forward,raissi2017physicsi,raissi2017physicsii,Raissi2019PINN, uchiyama1993solving, zhang2019learning, zhang2019quantifying,fan2019solving,yan2019adaptive}. 
Specifically, some high-dimensional HJ PDEs have been numerically solved using neural networks \cite{darbon2019overcoming,Han2018Solving,hure2019some,Sirignano2018DGM}.
In \cite{Sirignano2018DGM}, the solution to HJ PDEs is approximated by a deep neural network whose loss function is the $l^2$ error of the PDE, the initial condition and the boundary condition on randomly sampled points in the domain. In \cite{Han2018Solving}, a neural network architecture is proposed to approximate a backward stochastic differential equation which computes the solution to a second order HJ PDE via an associated stochastic representation formula. 
In \cite{hure2019some}, Hur{\'e} {\it et al.} approximate the solution and its gradient using two neural networks at each discretized time step. After the neural networks at a larger time $t_{j+1}$ are trained, the neural networks at $t_j$ are trained with loss function given by the error of the stochastic representation formula.
In \cite{darbon2019overcoming}, a neural network architecture is proposed for representing the viscosity solution to certain high dimensional HJ PDEs without error.
In addition, C{\'a}rdenas and Gibou \cite{cardenas2020deep} use neural networks to compute the mean curvature of the implicit level set function, which is the solution to a specific HJ PDE called level set equation.

\bigbreak
\noindent
\textbf{Organization of this paper.} 
This paper investigates the connections between two neural network architectures shown in Figs.~\ref{fig: nn1} and \ref{fig: nn2} and the viscosity solution of some HJ PDEs whose initial data and Hamiltonian satisfy specific assumptions. In Sec.~\ref{sec:background}, we introduce basic concepts in finite dimensional convex analysis which will be used later in this paper. In Sec.~\ref{sec:mainresults}, we present the main results. To be specific, we propose two neural network architectures. The first architecture is analyzed in Sec.~\ref{sec:firststructure}, while the second one is analyzed in Sec.~\ref{sec:secondstructure}. 
Thms.~\ref{thm:1} and \ref{thm:2} state that the neural network architectures shown in Figs.~\ref{fig: nn1} and \ref{fig: nn2} represent viscosity solutions to the HJ PDEs with convex Hamiltonian $\HHJ$ and initial data $\JHJ$ satisfying certain assumptions. 
We provide several examples and illustrations after each theorem. Finally, a conclusion is drawn in Sec.~\ref{sec:conclusion}.

\section{Background}
\label{sec:background}
In this section, we introduce related concepts in convex analysis that will be used in this paper. We refer readers to Hiriart\textendash Urruty and Lemar\'echal \citep{hiriart2013convexI,hiriart2013convexII} and Rockafellar \cite{rockafellar1970convex} for comprehensive references on finite-dimensional convex analysis. For the notation, we use $\R^n$ to denote the $n$-dimensional Euclidean space, on which the Euclidean scalar product is denoted by $\left\langle \cdot,\cdot\right\rangle$.

\begin{defn} (Convex sets and the unit simplex) A set $C\subset\mathbb{R}^{n}$ is called convex if for any $\lambda\in[0,1]$ and any $\boldsymbol{x},\boldsymbol{y}\in C$, the element $\lambda\boldsymbol{x} + (1-\lambda)\boldsymbol{y}$ is in $C$. The unit simplex is a specific convex set in $\R^n$, denoted by $\unitsim_n$, defined by
\begin{equation}\label{eqt:def_unitsimplex}
    \unitsim_n \coloneqq \left\{(\alpha_1, \dots, \alpha_n)\in [0,1]^n:\ \sum_{i=1}^n \alpha_i = 1\right\}.
\end{equation}
\end{defn} 

\begin{defn}
\label{def:domains_prop}(Domains and proper functions) The domain of a function $f\colon\mathbb{R}^{n}\to\mathbb{R}\cup\{+\infty\}$ is the set
\[
\dom f=\left\{ \boldsymbol{x}\in\mathbb{R}^{n}:f(\boldsymbol{x})<+\infty\right\} .
\]
A function $f\colon \R^n\to\R\cup\{+\infty\}$ is called proper if its domain is non-empty. 
\end{defn}

\begin{defn}
\label{def:convex}(Convex functions, concave functions and lower semicontinuity)
A proper function $f\colon\mathbb{R}^{n}\to\mathbb{R}\cup\{+\infty\}$ is called convex if the set $\dom f$ is convex and if for any $\boldsymbol{x},\boldsymbol{y}\in\dom f$ and all $\lambda\in[0,1]$, there holds
\begin{equation*}
f(\lambda\boldsymbol{x}+(1-\lambda)\boldsymbol{y})\leqslant\lambda f(\boldsymbol{x})+(1-\lambda)f(\boldsymbol{y}).\label{eq:convex_def}
\end{equation*}
A function $f\colon\mathbb{R}^{n}\to\mathbb{R}\cup\{-\infty\}$ is called concave if $-f$ is a convex function.
A proper function $f\colon\mathbb{R}^{n}\to\mathbb{R}\cup\{+\infty\}$ is called lower semicontinuous if for every sequence $\left\{ \boldsymbol{x}_{k}\right\} _{k=1}^{+\infty}$ in $\mathbb{R}^{n}$ with $\lim_{k\to+\infty}\boldsymbol{x}_{k}=\boldsymbol{x}\in\mathbb{R}^{n}$, we have $\liminf_{k\to+\infty}f(\boldsymbol{x}_{k})\geqslant f(\boldsymbol{x})$.
The class of proper, lower semicontinuous convex functions is denoted by $\Gamma_{0}(\mathbb{R}^{n})$.
\end{defn}

\begin{defn}
\label{def:legendre_t}(Fenchel--Legendre transform) Let $f\in\Gamma_{0}(\mathbb{R}^{n})$.
The Fenchel--Legendre transform $f^{*}\colon\mathbb{R}^{n}\to\mathbb{R}\cup\{+\infty\}$ of $f$ is defined as
\begin{equation*}
f^{*}(\boldsymbol{p})=\sup_{\boldsymbol{x}\in\mathbb{R}^{n}}\left\{ \left\langle \boldsymbol{p},\boldsymbol{x}\right\rangle -f(\boldsymbol{x})\right\} .\label{eq:fenchel_t_def}
\end{equation*}
For any $f\in\Gamma_{0}(\mathbb{R}^{n})$, the mapping $f\mapsto f^{*}$ is one-to-one. Moreover, there hold $f^{*}\in\Gamma_{0}(\mathbb{R}^{n})$ and $(f^{*})^{*}=f$. 
\end{defn}

\begin{defn} \label{def:infconv}(Inf-convolution)
Let $f,g\colon \R^n\to \R\cup\{+\infty\}$ be two proper convex functions satisfying 
\begin{equation}\label{eqt:infconv_assump}
    f(\bx) \geq \langle \bp, \bx\rangle + a  \quad\text{ and } \quad 
    g(\bx) \geq \langle \bp, \bx\rangle + a \text{ for every } \bx\in \R^n, 
\end{equation}
for some $\bp\in \R^n$ and $a\in \R$. The inf-convolution of $f$ and $g$, denoted by $f\square g$, is defined by 
\begin{equation*}
    f\square g(\bx) = \inf_{\bu\in\R^n} \{f(\bu) + g(\bx-\bu)\}.
\end{equation*}
Moreover, the function $f\square g \colon \R^n\to\R\cup\{+\infty\}$ is a proper and convex function \cite[Prop.~IV.2.3.2]{hiriart2013convexI}.
\end{defn}

\begin{defn}
\label{def:asymptoticFunction}
(Asymptotic function)
Let $f$ be a function in $\gmRn$ and $\bx_0$ be an arbitrary point in $\dom f$. The asymptotic function of $f$, denoted by $f'_\infty$, is defined by 
\begin{equation}\label{eqt:defasym}
    f'_\infty(\bd) = \sup_{s>0} \frac{f(\bx_0 + s\bd) - f(\bx_0)}{s} = \lim_{s\to+\infty} \frac{f(\bx_0 + s\bd) - f(\bx_0)}{s},
\end{equation}
for every $\bd\in \R^n$. In fact, this definition does not depend on the point $\bx_0$.
Moreover, the asymptotic function $f'_\infty$ is convex and positive $1-$homogeneous, i.e., $f'_\infty(\alpha\bd) = \alpha f'_\infty(\bd)$ for every $\alpha >0$ and $\bd\in\R^n$. For details, see \cite[Chap.~IV.3.2]{hiriart2013convexI}
\end{defn}

We summarize some notations and definitions in Tab. \ref{tab:notations}.
\newcolumntype{s}{>{\hsize=.7\hsize}X}
\begin{table}[!h]
\centering
 \caption{Notations used in this paper. Here, we use $f,g$ to denote functions from $\R^n$ to $\R\cup\{+\infty\}$ and $\bx,\by,\bp,\bd$ to denote vectors in $\Rn$. For simplicity, we omit the assumptions in the definitions.}
 \label{tab:notations}
\begin{tabularx}{\textwidth}{ l|s|X }
\hline
\noalign{\smallskip}
 Notation & Meaning & Definition\\
 \noalign{\smallskip}
 \hline
 \noalign{\smallskip}

 $\left\langle \cdot,\cdot\right\rangle $ & 
 Euclidean scalar product in $\mathbb{R}^{n}$ & $\langle \bx, \by\rangle \coloneqq \sum_{i=1}^n x_iy_i$
 \\
 $\unitsim_n$ & The unit simplex in $\R^n$ & $\left\{(\alpha_1, \dots, \alpha_n)\in [0,1]^n:\ \sum_{i=1}^n \alpha_i = 1\right\}$
 \\
 $\dom f$ & The domain of $f$ & $\{\bx\in \Rn:\ f(\bx) < +\infty\}$
 \\
 $\gmRn$ & A useful and standard class of convex functions & The set containing all proper, convex, lower semicontinuous functions from $\Rn$ to $\R\cup\{+\infty\}$
 \\
 $f^*$ & Fenchel--Legendre transform of $f$ & $f^*(\bp) \coloneqq \sup_{\bx\in \Rn} \{\langle \bp, \bx\rangle - f(\bx)\}$
 \\
 $f\square g$ & Inf-convolution of $f$ and $g$ &
 $f\square g(\bx) = \inf_{\bu\in\R^n} \{f(\bu) + g(\bx-\bu)\}$
 \\
 $f'_\infty$ & The asymptotic function of $f$ &
 $f'_\infty(\bd) = \sup_{s>0} \left\{\frac{1}{s}(f(\bx_0 + s\bd) - f(\bx_0))\right\}$
 \\
 \noalign{\smallskip}
\hline
 \end{tabularx}
\end{table}

\section{Main Results} \label{sec:mainresults}
In this paper, we consider the HJ PDE given by
\be \label{eqt:HJ}
\begin{dcases} 
\frac{\partial S}{\partial t}(\bx,t)+\HHJ(\nabla_{\bx}S(\bx,t))= 0 & \mbox{{\rm in} }\mathbb{R}^{n}\times(0,+\infty),\\
S(\bx,0)=\JHJ(\bx) & \mbox{{\rm in} }\mathbb{R}^{n},
\end{dcases}
\ee
where $H\colon \R^n\to \R\cup\{+\infty\}$ is called Hamiltonian, and $J\colon \R^n\to \R$ is the initial data. 
It is well-known that when $H$ is convex, the viscosity solution is given by the Lax-Oleinik formula \cite{bardi1984hopf, evans1998partial, Hopf1965} stated as follows
\be\label{eqt:LOformula}
\SLO(\bx,t) = \inf_{\bu\in \Rn} \left\{\JHJ(\bu) + t\HHJ^*\left(\frac{\bx - \bu}{t}\right)\right\} = \inf_{\bv\in\Rn} \left\{\JHJ(\bx-t\bv) + t\HHJ^*(\bv)\right\},
\ee
where $H^*$ is the Fenchel--Legendre transform of $H$.

In this part, we represent the Lax-Oleinik formula using two neural network architectures. The first one is given by
\begin{equation}\label{eqt:nn1}
    f_1(\bx, t) = \min_{i\in\{1,\dots,m\}} \left\{t\LNN\left(\frac{\bx-\bu_i}{t}\right) + a_i \right\}.
\end{equation}
In this function, $\{(\bu_i, a_i)\}_{i=1}^m\subset \Rn\times \R$ is the set of parameters, and the function $\LNN\colon \R^n\to \R$ is the activation function, which corresponds to the Lagrangian function in the Hamilton--Jacobi theory. An illustration is shown in Fig.~\ref{fig: nn1}.

\begin{figure}[ht]
\includegraphics[width=\textwidth]{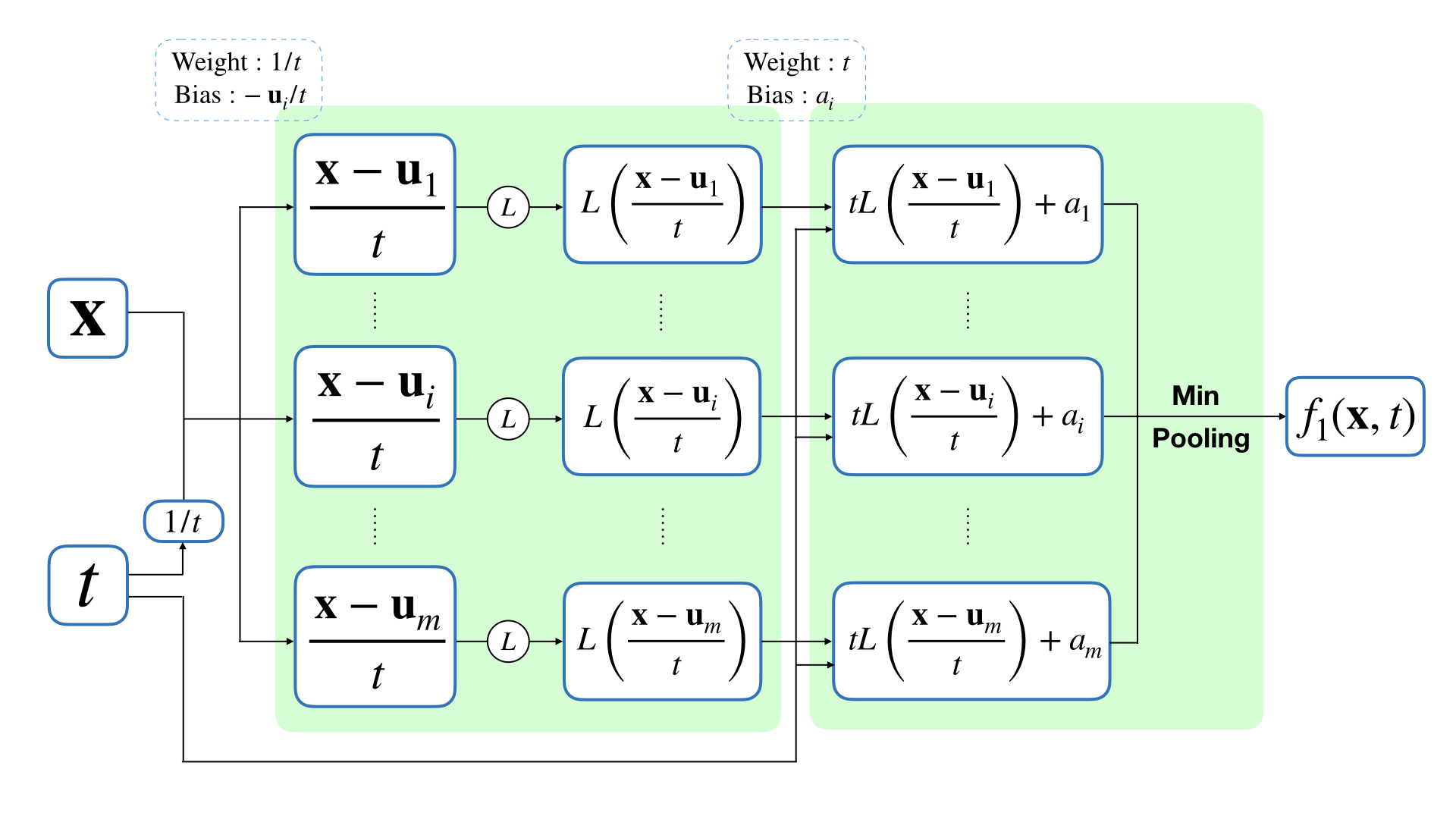}
\caption{An illustration of the architecture of the neural network (\ref{eqt:nn1}) that represents the Lax-Oleinik formula with specific initial condition $\JHJ = f_1(\cdot, 0)$ defined in \eqref{eqt:f1_t0} and the convex Hamiltonian $\HHJ = \LNN^*$. \label{fig: nn1}}
\end{figure}

The second neural network architecture is defined by
\begin{equation}\label{eqt:nn2}
    f_2(\bx,t) = \min_{i\in\{1,\dots,m\}} \left\{\JNN(\bx - t\bv_i) + tb_i\right\}.
\end{equation}
Here, $\{(\bv_i, b_i)\}_{i=1}^m\subset \Rn\times \R$ is the set of parameters, and $\JNN\colon \Rn\to \R$ is the activation function, which corresponds to the initial function in the HJ PDE. An illustration is shown in Fig.~\ref{fig: nn2}.

These two neural network architectures are further introduced and investigated in Section~\ref{sec:firststructure} and~\ref{sec:secondstructure}, respectively. To be specific, they are shown to represent a viscosity solution to certain HJ PDEs under some assumptions without errors. In addition, several examples are shown in each subsection. In these example, certain HJ PDEs are solved using corresponding neural network architectures. The Tensorflow codes of these two neural networks using our proposed architectures are provided in the website \url{https://github.com/TingweiMeng/NN_LO}.

\begin{figure}[ht]
\includegraphics[width=\textwidth]{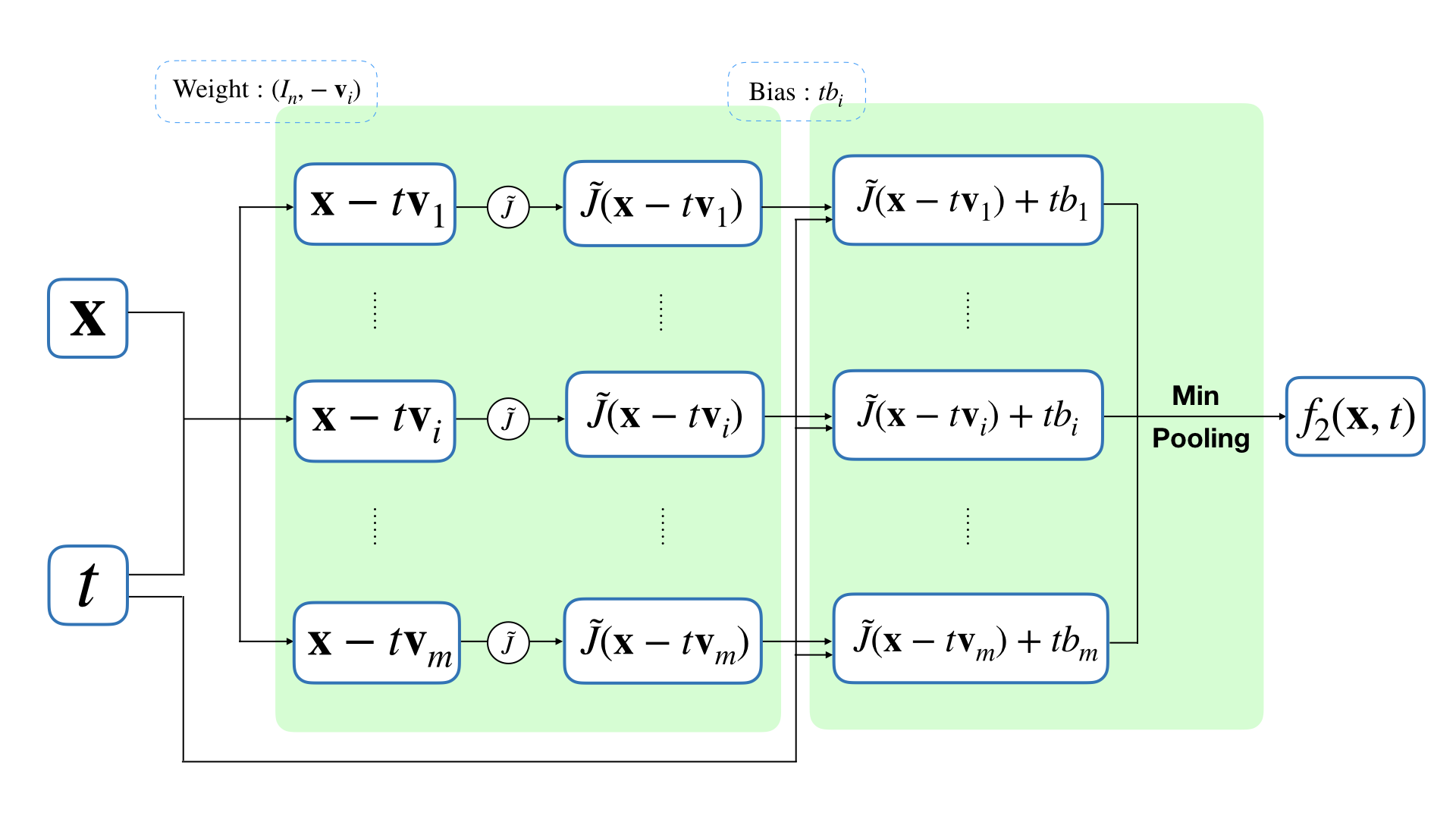}
\caption{An illustration of the architecture of the neural network (\ref{eqt:nn2}) that represents the Lax-Oleinik formula with specific initial condition $\JHJ = \JNN$ and the convex Hamiltonian $\HHJ$ defined in \eqref{eqt:thm2_H_def}. \label{fig: nn2}}
\end{figure}

\subsection{The first architecture}  \label{sec:firststructure}
In this subsection, we analyze the first neural network architecture given by Eq.~\eqref{eqt:nn1}. 
Before introducing the main theorem \ref{thm:1} in this subsection, we prove the following lemma which will be used in the proof of Thm.~\ref{thm:1}.
\begin{lem}\label{lem:conv_asymp}
Let $f$ be a function in $\gmRn$ and $f'_\infty$ be the asymptotic function of $f$.
Then, we have $f\square f'_\infty = f$.
\end{lem}
\begin{proof}
First we consider the case when $\bx\in\dom f$. By definition \ref{def:infconv} we have
\benn
(f\square f'_\infty) (\bx) = \inf_{\bu\in\Rn} \left\{f(\bu) + f'_\infty(\bx - \bu)\right\} \leq f(\bx) + f'_\infty(\mathbf{0}) = f(\bx),
\eenn
where the last equality holds because $f'_\infty(\mathbf{0}) = 0$ by definition \ref{def:asymptoticFunction}.
On  the other hand, taking $s=1$, $\bd = \bx - \bu$ and $\bx_0 = \bu$ in the second term in Eq.~\eqref{eqt:defasym} in definition \ref{def:asymptoticFunction}, we obtain
\be \label{eqt:1lem1}
f'_\infty(\bx-\bu) \geq f(\bu + \bx - \bu) - f(\bu) = f(\bx) - f(\bu),
\ee
for every $\bu \in \dom f$. As a result, we have
\benn
(f\square f'_\infty) (\bx) = \inf_{\bu\in\dom f} \left\{f(\bu) + f'_\infty(\bx - \bu)\right\}
\geq \inf_{\bu\in\dom f} \left\{f(\bu) + f(\bx) - f(\bu)\right\} = f(\bx).
\eenn
Therefore, we conclude that $(f\square f'_\infty) (\bx) = f(\bx)$ for every $\bx\in \dom f$.

Now we consider the case when $\bx \notin \dom f$ and prove $(f\square f'_\infty) (\bx) = +\infty$. It suffices to prove $f'_\infty(\bx - \bu) = +\infty$ for all $\bu \in \dom f$. Since $\bu\in \dom f$, Eq.~\eqref{eqt:1lem1} still holds. As a result, we have 
\benn
f'_\infty(\bx-\bu) \geq f(\bx) - f(\bu) = +\infty,
\eenn
since $\bx\notin\dom f$ and $\bu\in \dom f$. Therefore, we conclude that $(f\square f'_\infty) (\bx) = +\infty = f(\bx)$ for every $\bx \notin \dom f$.
\end{proof}

Now, we define the initial data $f_1(\cdot,0): \Rn \to \R$ as follows 
\begin{equation}\label{eqt:f1_t0}
    f_1(\bx,0) = \min_{i\in\{1,\dots, m\}} \left\{\LNN'_\infty(\bx-\bu_i) + a_i\right\},
\end{equation}
where $L'_\infty$ is the asymptotic function of $L$.
Then, we present the main theorem stating that the function $f_1$ solves the HJ PDE \eqref{eqt:HJ} with the initial condition given by $J = f_1(\cdot, 0)$ defined in \eqref{eqt:f1_t0} and the convex Hamiltonian $H$ which is the Fenchel--Legendre transform of $L$.

\begin{thm}\label{thm:1}
Let $L\colon \Rn \to \R$ be a convex uniformly Lipschitz function. Let $f_1$ be the function defined in \eqref{eqt:nn1}. Then $f_1 = \SLO$, where $\SLO$ is the Lax--Oleinik formula in \eqref{eqt:LOformula} with the initial condition $\JHJ = f_1(\cdot, 0)$ defined in \eqref{eqt:f1_t0} and the convex Hamiltonian defined by $\HHJ = \LNN^*$. Therefore, $f_1$ is a viscosity solution to the corresponding HJ PDE \eqref{eqt:HJ}.
\end{thm}

\begin{rem}
In the theorem above, we assume $L$ to be a convex uniform Lipschitz function, which implies that its Fenchel--Legendre transform $H$ has bounded domain, and hence $H$ may take the value $+\infty$ somewhere. As a result, the uniqueness theorem of the viscosity solution in \cite[Chap.~10.2]{evans1998partial} does not hold. To our knowledge, we are not aware of any uniqueness result of the viscosity solution to the HJ PDEs where $\dom H$ is bounded.
\end{rem}

\begin{proof}
Since $L$ is Lipschitz continuous, by \cite[Prop.~IV.3.2.7]{hiriart2013convexI} $L'_\infty$ is finite valued, which implies that $\Rn \ni \bx \mapsto f_1(\bx, 0)$ is finite valued  and  it is a valid initial condition.

Let $\bx\in\Rn$ and $t>0$. By definition \ref{def:infconv} and \eqref{eqt:f1_t0}, we have
\be \label{eqt:pfthm1_SLO}
\begin{split}
\SLO(\bx,t) &= \inf_{\bu\in\Rn} \left\{\JHJ(\bu) + tH^*\left(\frac{\bx-\bu}{t}\right)\right\}
= \inf_{\bu\in\Rn} \left\{\min_{i\in\{1,\dots,m\}} \left\{L'_\infty(\bu-\bu_i) + a_i\right\} + tH^*\left(\frac{\bx-\bu}{t}\right)\right\}\\
&= \min_{i\in\{1,\dots,m\}} \left\{a_i + \inf_{\bu\in\Rn}\left\{L'_\infty(\bu - \bu_i) + tH^*\left(\frac{\bx-\bu}{t}\right)\right\}\right\}\\
&= \min_{i\in\{1,\dots,m\}} \left\{a_i + \left(L'_\infty \square tH^*\left(\frac{\cdot}{t}\right)\right)(\bx - \bu_i)\right\}.
\end{split}
\ee
Since $L$ is convex with $\dom L = \R^n$, then the function $L$ is continuous \cite[Thm.~IV.3.1.2]{hiriart2013convexI}. As a result, $L$ is a function in $\gmRn$, hence we have $L = (L^*)^*$, which equals $H^*$ because we assume $H = L^*$.
Let $t>0$ and $h: \R^n\to \R$ be defined by $h(\bx) = tH^*\left(\frac{\bx}{t}\right) = tL\left(\frac{\bx}{t}\right)$ for every $\bx\in\R^n$. Let $\bx_0$ be an arbitrary point in $\dom h$, which implies $\frac{\bx_0}{t}\in\dom L$. By definition \ref{def:asymptoticFunction}, the asymptotic function of $h$ evaluated at $\bd$ is given by 
\benn
\begin{split}
h'_\infty(\bd) = &\sup_{s>0} \left\{\frac{1}{s} \left(tL\left(\frac{\bx_0 + s\bd}{t}\right) - tL\left(\frac{\bx_0}{t}\right)\right)\right\}
= \sup_{s>0} \left\{\frac{t}{s} \left(L\left(\frac{\bx_0}{t} + \frac{s}{t}\bd\right) - L\left(\frac{\bx_0}{t}\right)\right)\right\}\\
= &\sup_{\tau > 0} \left\{\frac{1}{\tau} \left(L\left(\frac{\bx_0}{t} + \tau\bd\right) - L\left(\frac{\bx_0}{t}\right)\right)\right\}
= L'_\infty(\bd),
\end{split}
\eenn
where in the third equality we set $\tau = \frac{s}{t}$. Hence, using the equality above, the definition of $h$ and by invoking Lem.~\ref{lem:conv_asymp}, we obtain
\benn
\left(L'_\infty \square tH^*\left(\frac{\cdot}{t}\right)\right)(\bx - \bu_i)
= \left(h'_\infty \square h\right)(\bx - \bu_i)
= h(\bx - \bu_i) = tL\left(\frac{\bx - \bu_i}{t}\right).
\eenn
We combine the equality above with \eqref{eqt:pfthm1_SLO}, to obtain
\benn
\SLO(\bx,t) = \min_{i\in\{1,\dots,m\}} \left\{a_i + \left(L'_\infty \square tH^*\left(\frac{\cdot}{t}\right)\right)(\bx - \bu_i)\right\}
= \min_{i\in\{1,\dots,m\}} \left\{a_i + tL\left(\frac{\bx - \bu_i}{t}\right)\right\} = f_1(\bx,t).
\eenn
Therefore, we conclude that $\SLO(\bx,t) = f_1(\bx,t)$ for each $\bx\in\Rn$ and $t>0$. Then, using the same proof as in \cite[Sec.~10.3.4, Thm.~3]{evans1998partial}, we conclude that $f_1$ is a viscosity solution to the corresponding HJ PDE \eqref{eqt:HJ}.
\end{proof}

\begin{rem}
Although the initial conditions for the HJ PDE considered in Theorem~\ref{thm:1} are given by a representation formula~\eqref{eqt:f1_t0}, it is not as restricted as it may seem to be.
Indeed, the functions in the form of~\eqref{eqt:f1_t0} can approximate a meaningful initial condition when $m$ approaches infinity. We will illustrate this point using an example. Consider the Lagrangian function $L\colon\Rn\to\R$ satisfying $L'_\infty = \|\cdot\|$ (for instance when $L = \|\cdot\|$). Then the domain of the Hamiltonian $H$ is the unit ball in $\R^n$, denoted by $B_1(\R^n)$. For this Hamiltonian, the reasonable set of initial data $J$ is the set of $1-$Lipschitz functions. From the physics point of view, the initial momentum $\bp_0$ (given by the gradient of $J$ at the initial position $\bx_0$) needs to be in $\dom H = B_1(\R^n)$, in order to have a finite energy $H(\bp_0)$. Therefore, the initial data $J$ needs to be $1-$Lipschitz. Now we argue that any $1-$Lipschitz function can be approximated using functions in the form of~\eqref{eqt:f1_t0} when $m$ increases to infinity. As a result, any reasonable initial condition can be approximated using~\eqref{eqt:f1_t0}.
Let $g\colon \Rn\to\R$ be an arbitrary $1-$Lipschitz function. Let $\{\bu_i\}$ be a dense sequence in $\R^n$ and let $a_i := g(\bu_i)$ for each $i\in\{1,2,\dots\}$. Define $g_m\colon \R^n\to\R$ by the formula in~\eqref{eqt:f1_t0} with the chosen parameters $\{\bu_i,a_i\}_{i=1}^m$, i.e., define $g_m$ by
\begin{equation*}
    g_m(\bx):= \min_{i\in\{1,\dots,m\}} \{\|\bx-\bu_i\| + a_i\} = \min_{i\in\{1,\dots,m\}} \{\|\bx-\bu_i\| + g(\bu_i)\}.
\end{equation*}
It is straightforward to check that $g_m(\bu_i) = g(\bu_i)$ for each $i\in\{1,\dots, m\}$, and $g_m(\bx)\geq g(\bx)$ for each $\bx\in\R^n$, by using assumption that $g$ is $1-$Lipschitz. Therefore, $\{g_m\}$ is a decreasing sequence which is bounded below by $g$.
Then, it is straightforward to check that $g_m$ converges to $g$ pointwisely as $m$ going to infinity, by using the assumption that $\{\bu_i\}$ is dense in $\R^n$ and $\{g_m\}, g$ are $1-$Lipschitz. Moreover, this convergence can be improved to $\Gamma-$convergence, since we have the monotonicity $g_1 \geq g_2\geq \cdots \geq g$. Therefore, in this example, the set of the functions in the form of~\eqref{eqt:f1_t0} is actually dense (in the sense of pointwise convergence and $\Gamma-$convergence) in the set of $1-$Lipschitz functions, which is a reasonable set for the initial conditions to the HJ PDEs with this Lagrangian $L$, as we claimed above.
\end{rem}

\begin{example} \label{eg:f1_1}
Let us consider the following one dimensional example that  illustrates the function $f_1\colon \R\times [0,+\infty)\to \R$ with three neurons, i.e., we set $n=1$ and $m=3$. The Lagrangian $L$ is defined as follows
\begin{equation*}
    L(x) = \begin{dcases}
    -x - \frac{1}{2} & x < -1,\\
    \frac{x^2}{2} &-1\leq x\leq 2,\\
    2x - 2 & x >2,
    \end{dcases}
\end{equation*}
for each $x\in\R$. Then, by Thm.~\ref{thm:1}, the Hamiltonian $H$ is given by
\begin{equation*}
    H(p) = L^*(p) = 
    \begin{dcases}
    \frac{p^2}{2} & -1\leq p\leq 2,\\
    +\infty &\text{otherwise.}
    \end{dcases}
\end{equation*}
Also, by Thm.~\ref{thm:1}, the initial data $\JHJ$ is given by $f_1(\cdot, 0)$ defined in \eqref{eqt:f1_t0}. In other words, $\JHJ$ is the minimum of three functions, each of which is a shift of the function $L'_\infty$, which by definition \ref{def:asymptoticFunction} reads as follows
\begin{equation*}
    L'_\infty(x) = \begin{dcases}
    -x &x<0,\\
    2x &x\geq 0.
    \end{dcases}
\end{equation*}
In this example, we choose the parameters $(u_1, a_1) = (-2, -0.5)$, $(u_2, a_2) = (0, 0)$ and $(u_3, a_3) = (2,-1)$. The corresponding functions $\JHJ$, $\HHJ$ and $f_1$ are shown in Fig.~\ref{fig: egf1}, where (a) shows the initial value $\JHJ$, (b) shows the convex Hamiltonian $\HHJ$, and (c) and (d) show the solution $S = f_1$ evaluated at $t = 1$ and $t = 3$, respectively. Note that our proposed architecture computes the viscosity solution without numerical errors. The viscosity solution in this example is not a classical solution, and there exist points where the solution is not differentiable. 
In Fig.~\ref{fig: egf1}~(c) and~(d), we can observe kinks in the graph of the functions given by our proposed neural network architecture. It can be seen from the non-smoothness of the graphs in Fig.~\ref{fig: egf1}~(c) and~(d) that our proposed architecture computes the viscosity solution without any numerical smoothing effect.

\begin{figure}[ht]
\begin{minipage}[b]{.49\linewidth}
  \centering
    \includegraphics[width=.99\textwidth]{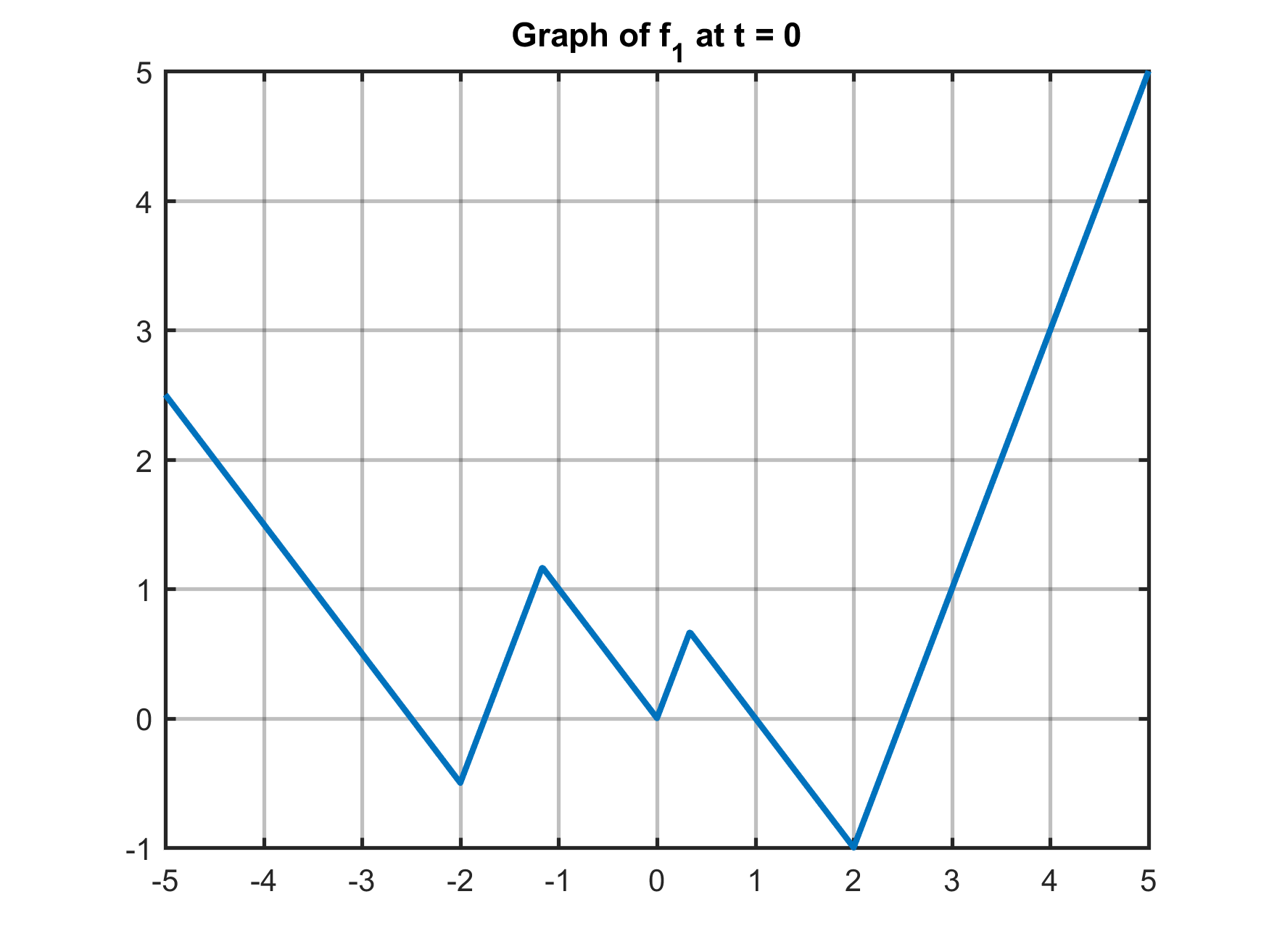}
  \centerline{\footnotesize{(a)}}\medskip
\end{minipage}
\hfill
\begin{minipage}[b]{.49\linewidth}
  \centering
\includegraphics[width=.99\textwidth]{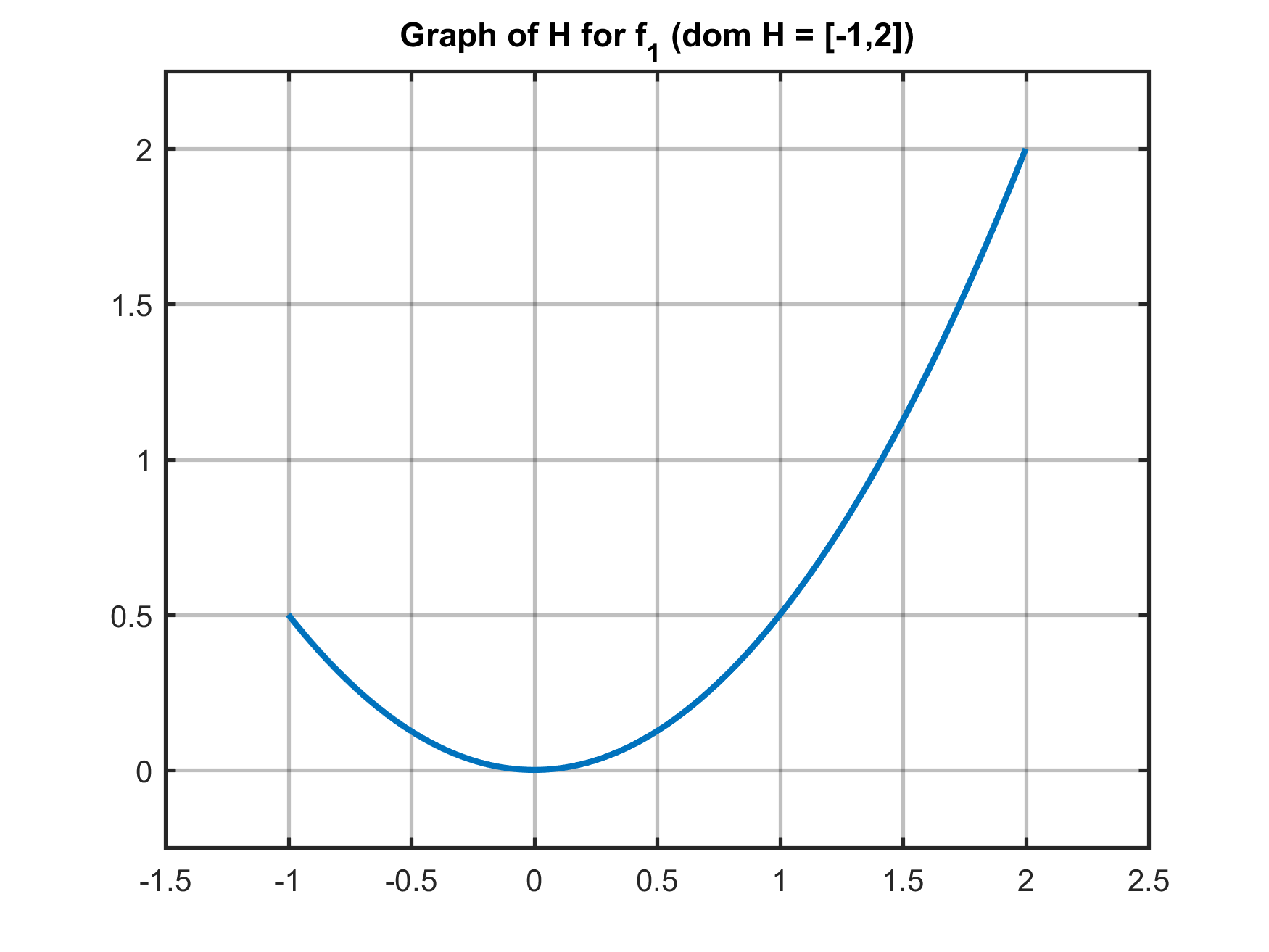}
  \centerline{\footnotesize{(b)}}\medskip
\end{minipage}
%%%% second row
\begin{minipage}[b]{.49\linewidth}
  \centering
    \includegraphics[width=.99\textwidth]{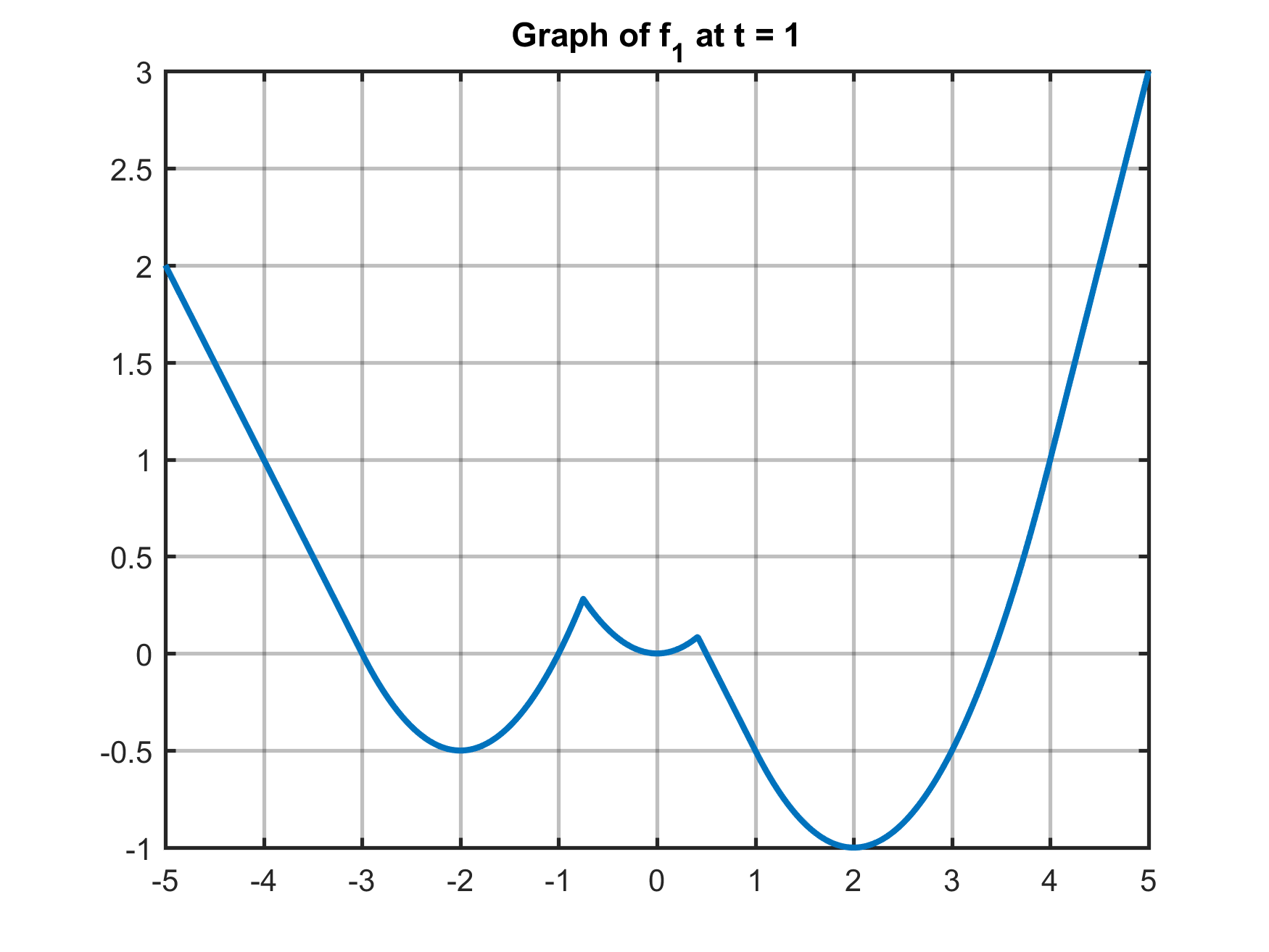}
  \centerline{\footnotesize{(c)}}\medskip
\end{minipage}
\hfill
\begin{minipage}[b]{.49\linewidth}
  \centering
    \includegraphics[width=.99\textwidth]{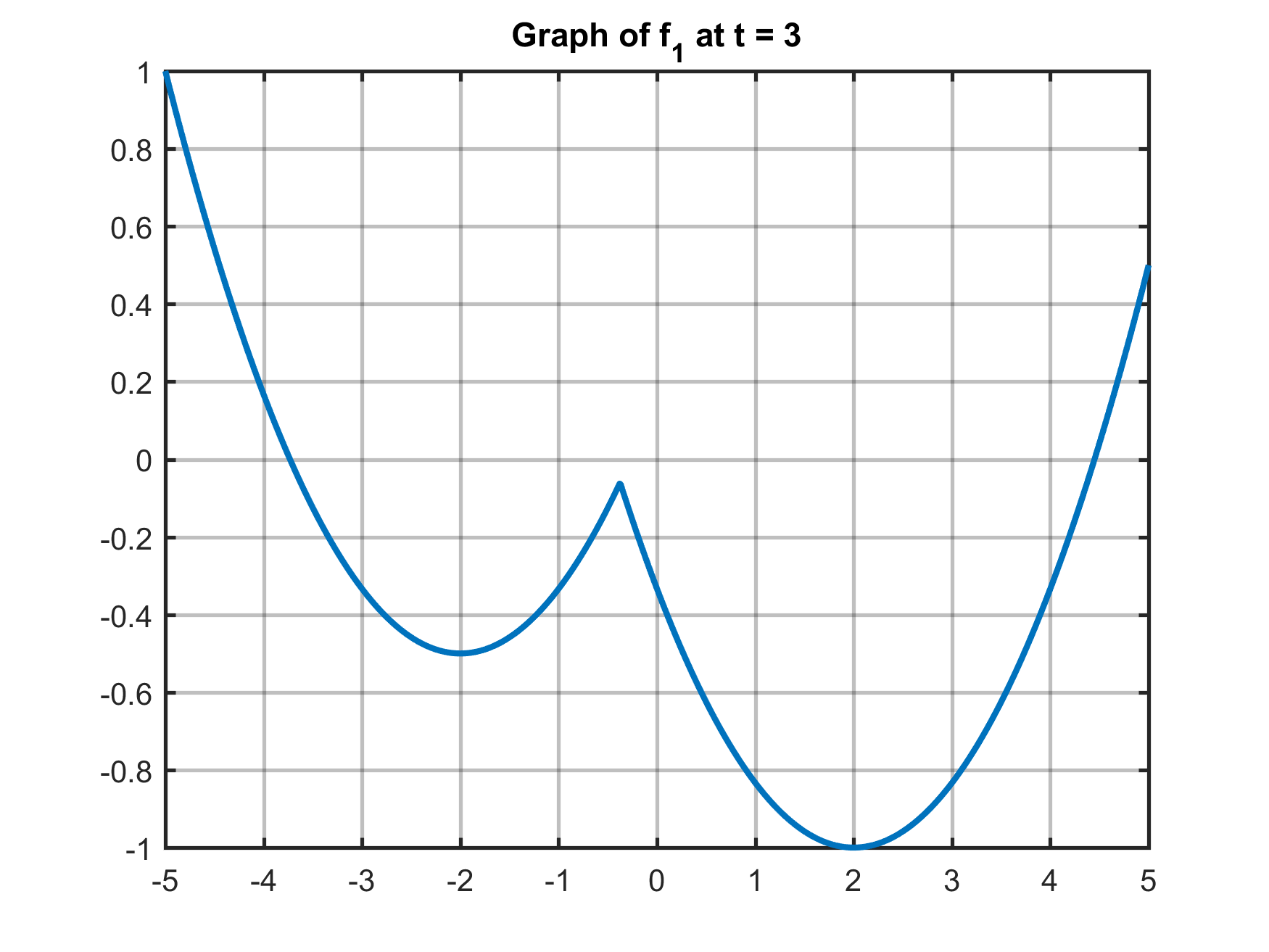}
  \centerline{\footnotesize{(d)}}\medskip
\end{minipage}
\caption{The graph of $f_1$ in example \ref{eg:f1_1}. The figures (a) and (b) show the initial value $\JHJ$ and the Hamiltonian $\HHJ$, respectively. The figures (c) and (d) show the solution $S = f_1$ evaluated at $t = 1$ and $t = 3$, respectively.
\label{fig: egf1}}
\end{figure}
\end{example}

\begin{example} \label{eg:f1_2}
We now present a high dimensional example. To be specific, the dimension is set to be $n=10$, and the solution $f_1: \R^{10}\times [0,+\infty) \to \R$ is represented by a neural network with three neurons, i.e., $m=3$. The activation function $L$ is given by
\begin{equation*}
    L(\bx) = \max\{ \|\bx\|_2 - 1, 0\} = \begin{dcases}
    \|\bx\|_2-1 & \text{ if } \|\bx\|_2 >1,\\
    0 & \text{ if } \|\bx\|_2\leq 1.
    \end{dcases}
\end{equation*}
The corresponding Hamiltonian is given by
\begin{equation*}
    H(\bp) = L^*(\bp) = \begin{dcases}
\|\bp\|_2 & \text{ if }\|\bp\|_2\leq 1,\\
+\infty & \text{ if } \|\bp\|_2 > 1.
\end{dcases}
\end{equation*}
The parameters are chosen to be $\bu_1 = (-2,0,0,0,\dots,0)$, $\bu_2 = (2,-2,-1, 0,\dots,0)$, $\bu_3 = (0,2,0,0,\dots,0)$, $a_1 = -0.5$, $a_2 = 0$ and $a_3 = -1$. 

By definition \ref{def:asymptoticFunction} and straightforward computation, we obtain
$L'_\infty(\bd) = \|\bd\|_2$.
Hence, the initial condition for the corresponding HJ PDE is given by Eq.~\eqref{eqt:f1_t0}, which in this example reads
\begin{equation*}
J(\bx) = \min_{i \in \{1,2,3\}} \{\|\bx - \bu_i\|_2 + a_i\}.
\end{equation*}

The accompanying figure~\ref{fig: egf1_hd} shows the graph of of $f_1$ for a  2-dimensional slice.
To be specific, we fix $\bx = (x_1, x_2, 0, \dots, 0)$, and compute $f_1(\bx, t)$ at $t= 10^{-6}$, $1$, $3$ and $5$. Note that the formula \eqref{eqt:nn1} is not well-defined for $t=0$, hence we use a small number $10^{-6}$ instead.
In each figure, the color is given by the function value $f_1(\bx,t)$ and the x and y axes represent the variables $x_1$ and $x_2$, respectively. 
The solutions evaluated at $t=10^{-6}$, $t=1$, $t=3$ and $t=5$ are shown in (a), (b), (c) and (d), respectively. The viscosity solution in this example is not a classical solution.
Note that there are several kinks on some level curves of the solution in each figure in Fig.~\ref{fig: egf1_hd}. Recall that the non-smoothness of the level curves implies the non-smoothness of the function. It can be seen from the non-smoothness of the level curves in Fig.~\ref{fig: egf1_hd} that our proposed architecture computes the viscosity solution without any numerical smoothing effect.

\begin{figure}[ht]
\begin{minipage}[b]{.49\linewidth}
  \centering
  \includegraphics[width=.99\textwidth]{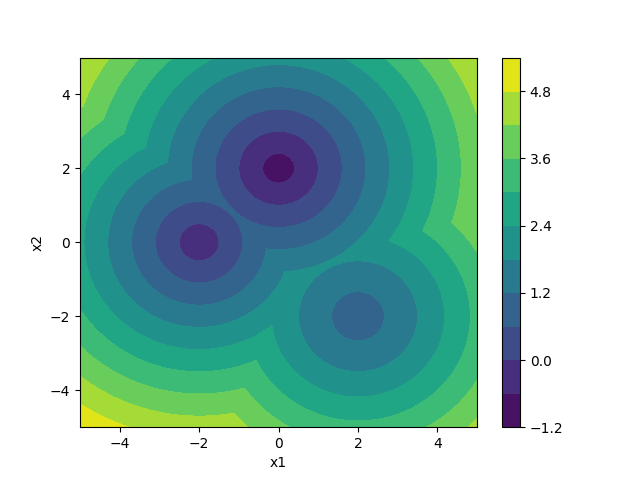}
  \centerline{\footnotesize{(a)}}\medskip
\end{minipage}
\hfill
\begin{minipage}[b]{.49\linewidth}
  \centering
  \includegraphics[width=.99\textwidth]{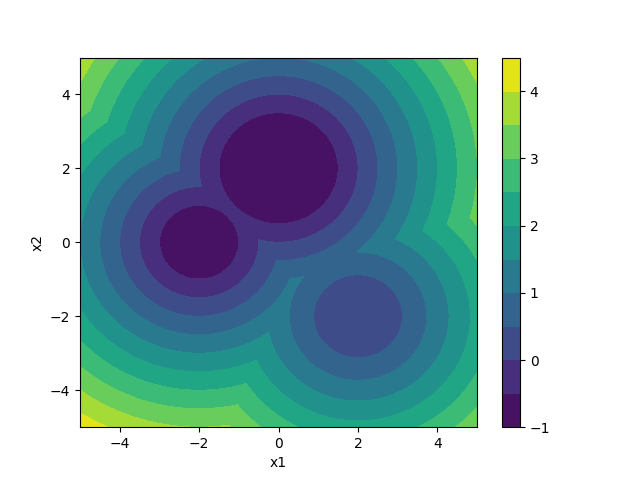}
  \centerline{\footnotesize{(b)}}\medskip
\end{minipage}
%%%% second row
\begin{minipage}[b]{.49\linewidth}
  \centering
    \includegraphics[width=.99\textwidth]{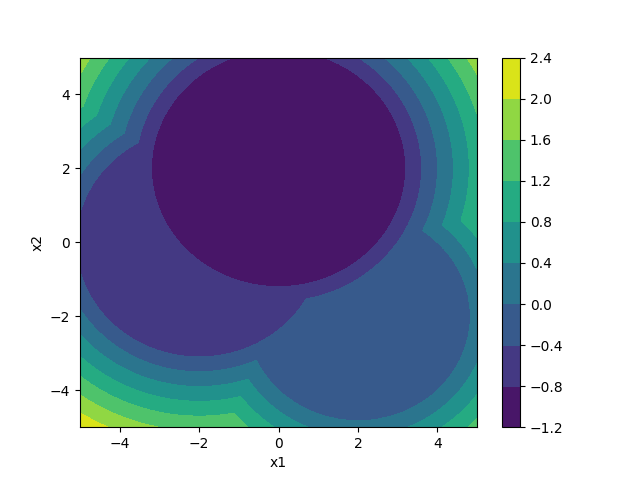}
  \centerline{\footnotesize{(c)}}\medskip
\end{minipage}
\hfill
\begin{minipage}[b]{.49\linewidth}
  \centering
    \includegraphics[width=.99\textwidth]{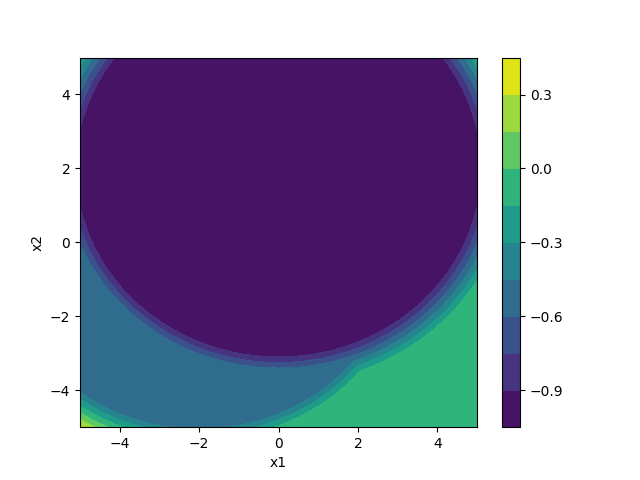}
  \centerline{\footnotesize{(d)}}\medskip
\end{minipage}
\caption{A two dimensional slice of
the graph of $f_1$ in example~\ref{eg:f1_2}. In each figure, the x and y axes correspond to the variables $x_1$ and $x_2$, which are the first and second coordinates of the variable $\bx = (x_1, x_2, 0,\dots, 0)$.
The color is given by the function value $f_1(\bx, t)$.
The figures (a), (b), (c) and (d) show contour lines of the solution $f_1(\bx, t)$ at $t=10^{-6}$, $t=1$, $t=3$ and $t=5$, respectively. 
\label{fig: egf1_hd}}
\end{figure}

\end{example}
\subsection{The second architecture} \label{sec:secondstructure}
In this part, we analyze the second neural network architecture given by Eq.~\eqref{eqt:nn2}.
Here, we assume the parameters $\{(\bv_i, b_i)\}_{i=1}^m$ satisfy the following assumption
\begin{itemize}
    \item[(H)] There exists a convex function $\ell\colon \Rn\to \R$ satisfying $\ell(\bv_i) = b_i$ for all $i\in \{1,\dots, m\}$.
\end{itemize}
Under this assumption, we present the following main theorem which states that the second architecture gives a viscosity solution to the corresponding HJ PDE, where the initial data is given by the activation function $\JNN$ in the neural network, and the Hamiltonian is a convex piecewise affine function determined by the parameters $\{(\bv_i,b_i)\}_{i=1}^m$.
\begin{thm}\label{thm:2}
Assume the function $\JNN\colon \Rn\to \R$ is a concave function and the assumption (H) is satisfied. Let $f_2$ be the function defined in \eqref{eqt:nn2}. Then $f_2 = \SLO$, where $\SLO$ is the Lax--Oleinik formula defined by \eqref{eqt:LOformula} with initial condition $\JHJ = \JNN$ and the Hamiltonian $\HHJ$ defined by 
\be\label{eqt:thm2_H_def}
\HHJ(\bp) = \max_{i\in\{1,\dots, m\}} \left\{ \langle \bp, \bv_i\rangle - b_i \right\}, 
\ee
for every $\bp\in\Rn$. Hence $f_2$ is a concave viscosity solution to the corresponding HJ PDE \eqref{eqt:HJ}.
\end{thm}
\begin{proof}
By assumption (H) and simply changing the notations in \cite[Lem.~3.1]{darbon2019overcoming}, we have
\be\label{eqt:thm2_Hstar}
H^*(\bv)= \min\left\{\sum_{i=1}^m \alpha_i b_i \colon (\alpha_1,\dots,\alpha_m)\in\unitsim_m, \, \sum_{i=1}^m \alpha_i\bv_i =\bv\right\},
\ee
for each $\bv\in\co\{\bv_1,\dots, \bv_m\} = \dom H^*$, where $\unitsim_m$ is the unit simplex defined in \eqref{eqt:def_unitsimplex}. Also, we have $H^*(\bv_k) = b_k$ for each $k\in\{1, \dots, m\}$.

For each $\bx\in \Rn$, $t>0$ and $\bv\in\co\{\bv_1,\dots, \bv_m\}$, let $\boldsymbol{\alpha} = (\alpha_1,\dots, \alpha_m)\in \unitsim_m$ be the minimizer in the minimization problem in \eqref{eqt:thm2_Hstar} evaluated at $\bv$. In other words, we have 
\begin{equation}\label{eqt:thm2_eqts}
    \sum_{i=1}^m \alpha_i = 1, \quad
    \sum_{i=1}^m \alpha_i\bv_i =\bv, \quad
    \sum_{i=1}^m \alpha_i b_i = H^*(\bv), \quad 
    \text{ and }\alpha_j\in[0,1] \text{ for each } j\in \{1,\dots,m\}.
\end{equation}
Then, by \eqref{eqt:thm2_eqts} and the assumption that $\JHJ = \JNN$ is concave, we have
\benn
\begin{split}
&\JHJ(\bx - t\bv) + tH^*(\bv)
= \JHJ\left(\sum_{i=1}^m\alpha_i\left(\bx - t\bv_i\right)\right) + t\sum_{i=1}^m\alpha_i b_i
\geq \sum_{i=1}^m \alpha_i\JHJ\left(\bx - t\bv_i\right) + \sum_{i=1}^m\alpha_i tb_i\\
=& \sum_{i=1}^m \alpha_i(\JHJ\left(\bx - t\bv_i\right) + tb_i)
\geq \min_{i\in\{1,\dots,m\}}\left\{\JHJ\left(\bx - t\bv_i\right) +tb_i\right\} = f_2(\bx,t).
\end{split}
\eenn
As a result, we conclude that
\benn
\SLO(\bx,t) = \inf_{\bv\in\dom H^*} \left\{\JHJ(\bx - t\bv) + tH^*(\bv)\right\}
\geq f_2(\bx,t).
\eenn
On the other hand, recall that $b_k = H^*(\bv_k)$ for each $k\in\{1,\dots, m\}$, hence we obtain
\benn
f_2(\bx,t) = \min_{i\in\{1,\dots,m\}}\left\{\JHJ\left(\bx - t\bv_i\right) +tH^*(\bv_i)\right\}\geq 
\inf_{\bv\in\Rn} \left\{\JHJ(\bx - t\bv) + tH^*(\bv)\right\}
= \SLO(\bx,t).
\eenn
Therefore, we conclude that $f_2(\bx,t) = \SLO(\bx,t)$ for each $\bx\in\Rn$ and $t>0$. 

Note that $H$ is a convex function, since it is the maximum of affine functions. Then, by the same proof as in \cite[Sec.~10.3.4, Thm.~3]{evans1998partial}, we conclude that $f_2$ is a viscosity solution to the corresponding HJ PDE. Moreover, since $\JNN$ is concave, $f_2$ is the minimum of concave functions, which implies the concavity of $f_2$.
\end{proof}

\begin{rem}
In the second architecture, if we furthermore assume that the initial condition $\JHJ=\JNN$ is uniformly Lipschitz, then $f_2$ is the unique uniformly continuous viscosity solution to the corresponding HJ PDE. This conclusion directly follows from \cite[Thm.~2.1]{bardi1984hopf}.
\end{rem}

\begin{example} \label{eg:f2_1d}
Here, we provide a one dimensional example of the function $f_2$. To be specific, we consider $f_2\colon \R\times [0,+\infty)\to \R$ represented by the neural network in Fig.~\ref{fig: nn2} with three neurons, i.e., we set $n=1$ and $m=3$. The initial value is given by $\JHJ(x) = -\frac{x^2}{2}$ for each $x\in\R$, and the Hamiltonian $H$ is given by the piecewise affine function in Eq.~\eqref{eqt:thm2_H_def} with $(v_1,b_1) = (-2,0.5)$, $(v_2,b_2) = (0,-5)$ and $(v_3,b_3) = (2,1)$. The functions $\JHJ$, $\HHJ$ and $f_2$ are shown in Fig.~\ref{fig: egf2}, where (a) shows the initial value $\JHJ$, (b) shows the convex Hamiltonian $\HHJ$, and (c) and (d) show the solution $S = f_2$ evaluated at $t = 1$ and $t = 3$, respectively. One can observe that there are several kinks on the graph of the solution shown in Fig.~\ref{fig: nn2}~(c) and~(d), which implies that the solution given by the proposed neural network architecture is not differentiable at these kinks. In other words, the proposed architecture provides the viscosity solution to the HJ PDE without any numerical smoothing effect.

\begin{figure}[ht]
\begin{minipage}[b]{.49\linewidth}
  \centering
  \includegraphics[width=.99\textwidth]{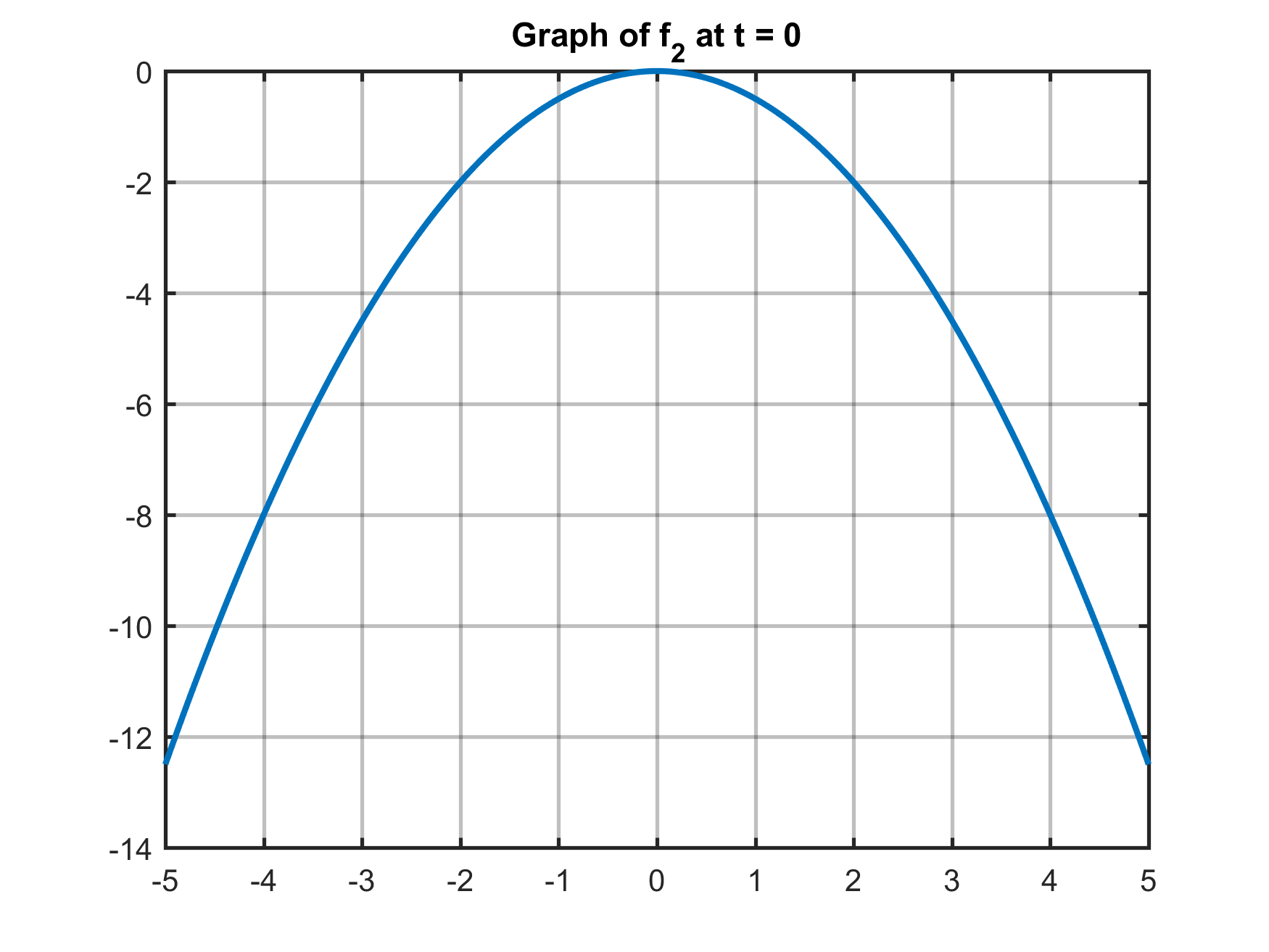}
  \centerline{\footnotesize{(a)}}\medskip
\end{minipage}
\hfill
\begin{minipage}[b]{.49\linewidth}
  \centering
  \includegraphics[width=.99\textwidth]{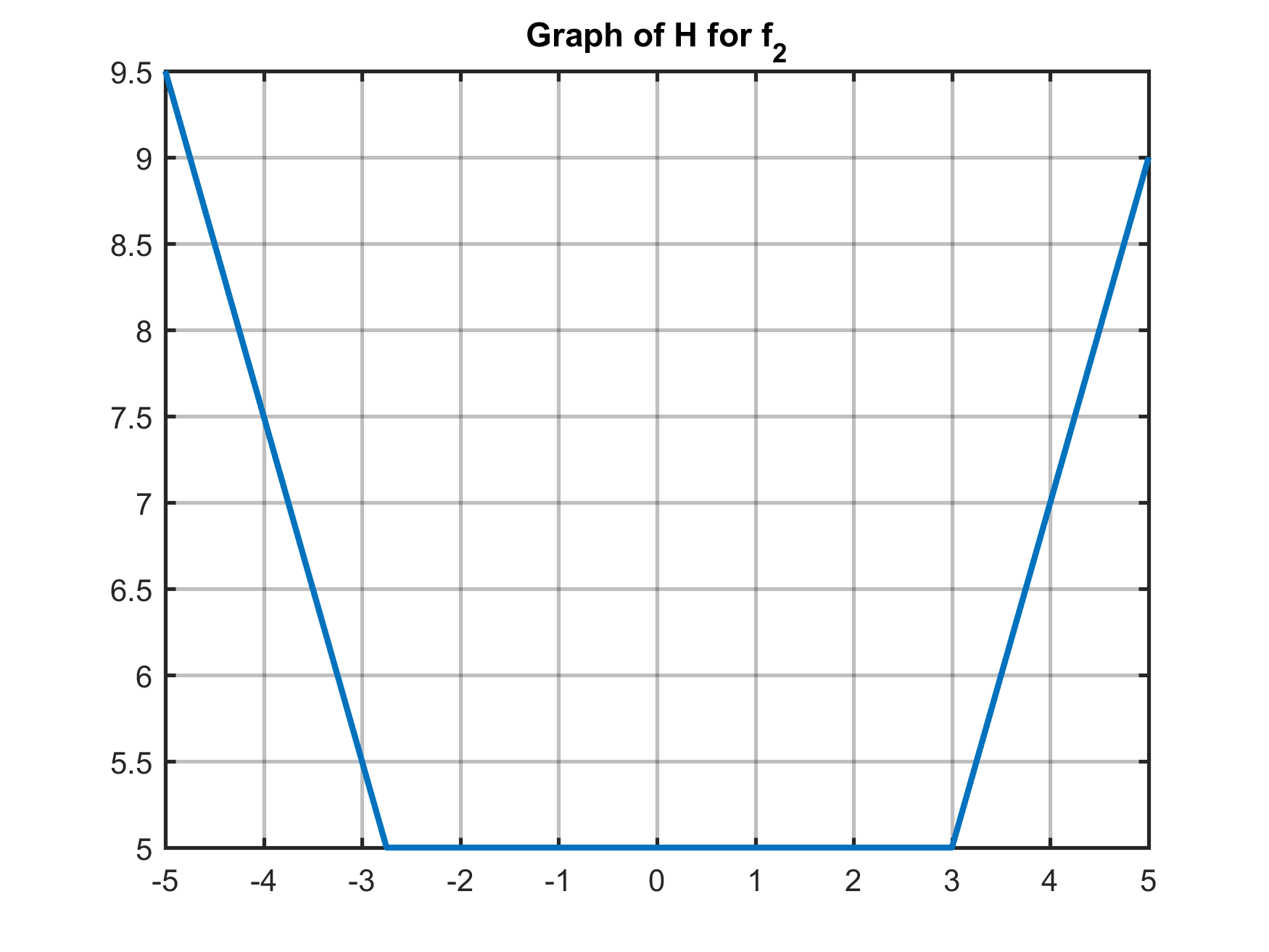}
  \centerline{\footnotesize{(b)}}\medskip
\end{minipage}
%%%% second row
\begin{minipage}[b]{.49\linewidth}
  \centering
    \includegraphics[width=.99\textwidth]{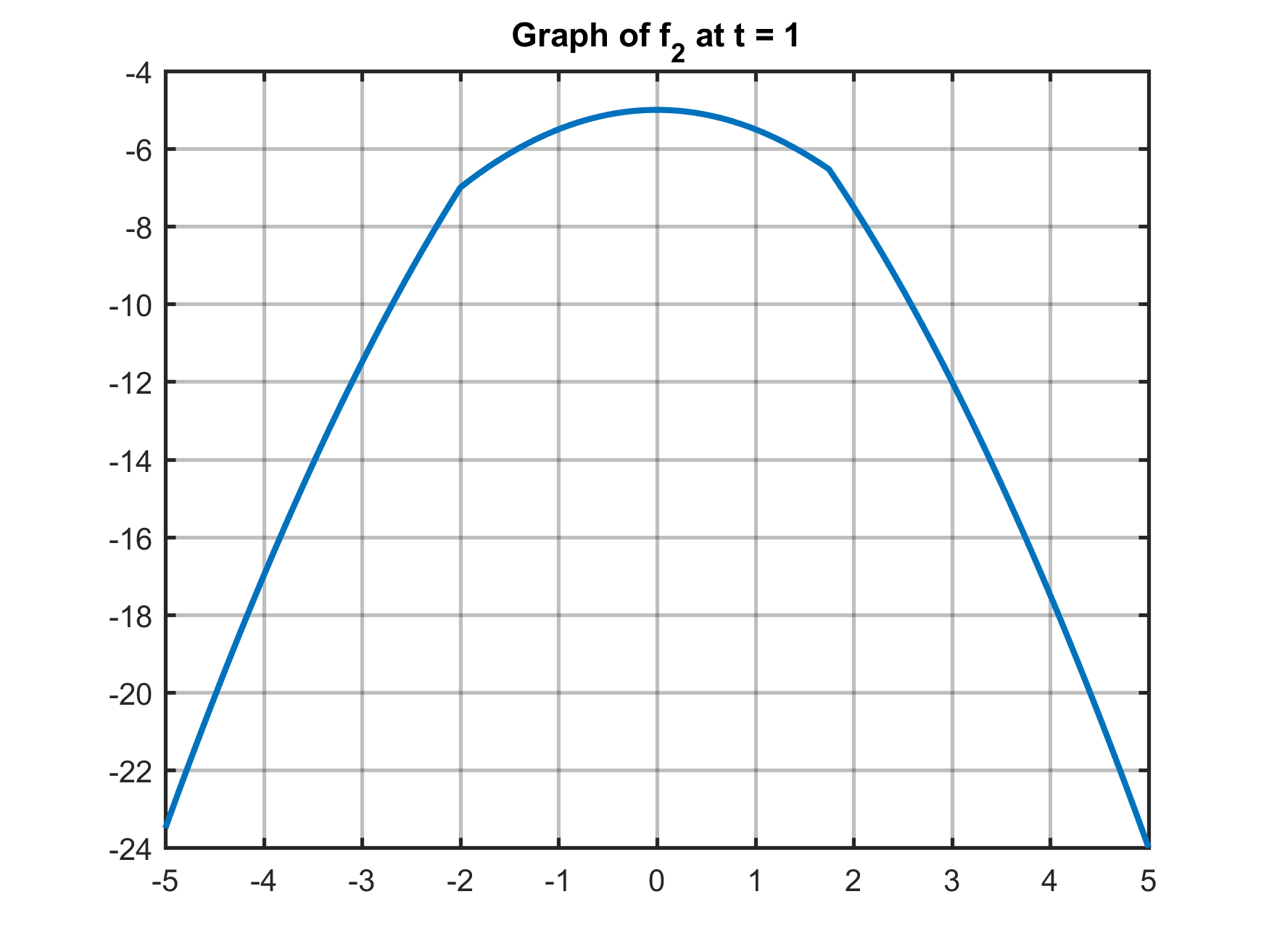}
  \centerline{\footnotesize{(c)}}\medskip
\end{minipage}
\hfill
\begin{minipage}[b]{.49\linewidth}
  \centering
    \includegraphics[width=.99\textwidth]{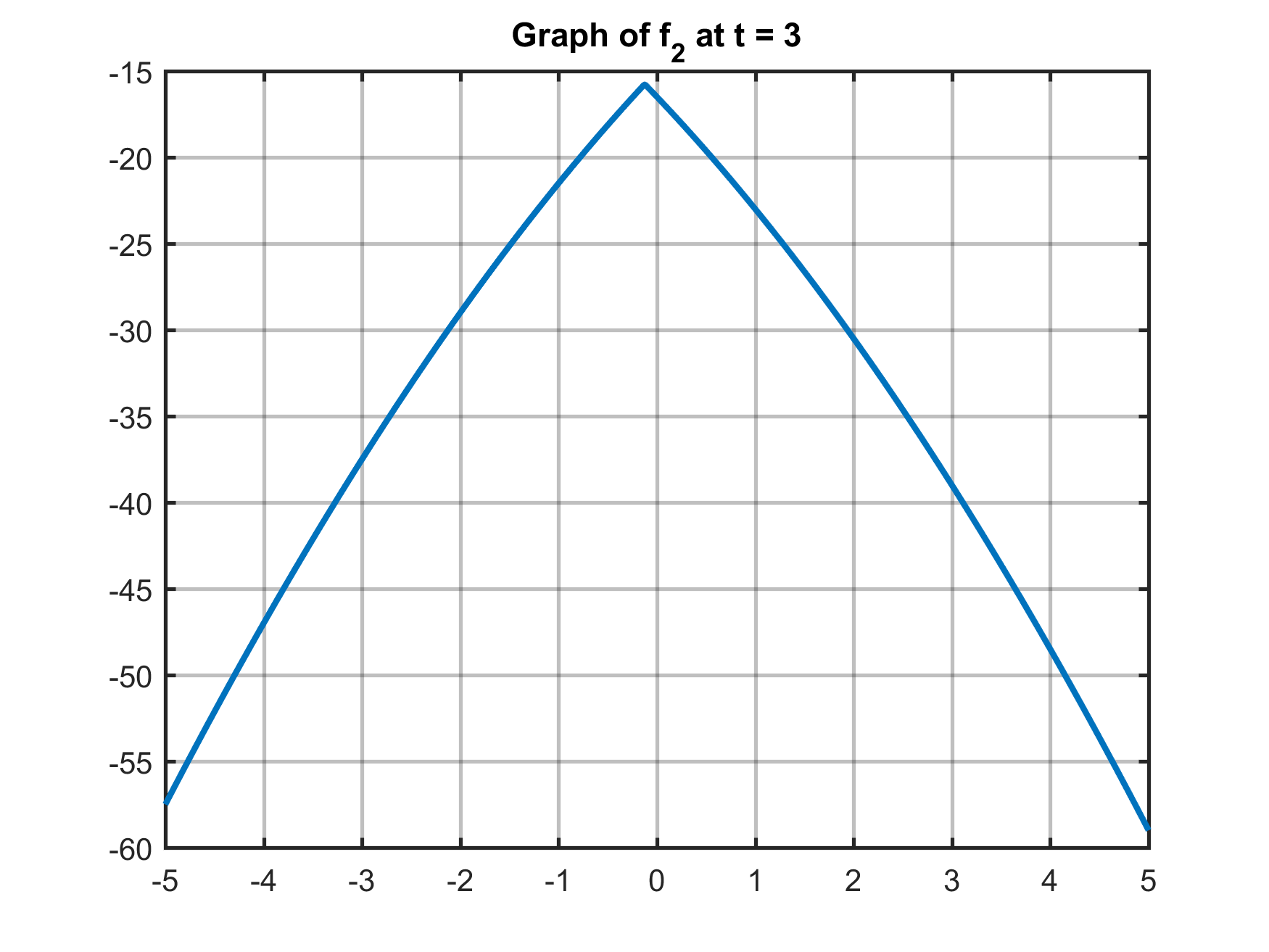}
  \centerline{\footnotesize{(d)}}\medskip
\end{minipage}
\caption{The graph of $f_2$ in example \ref{eg:f2_1d}. The figures (a) and (b) show the initial value $\JHJ$ and the Hamiltonian $\HHJ$, respectively. The figures (c) and (d) show the solution $S = f_2$ evaluated at $t = 1$ and $t = 3$, respectively. \label{fig: egf2}}
\end{figure}
\end{example}

\begin{example} \label{eg: f2_hd}

\begin{figure}[htbp]
\begin{minipage}[b]{.49\linewidth}
  \centering
  \includegraphics[width=.99\textwidth]{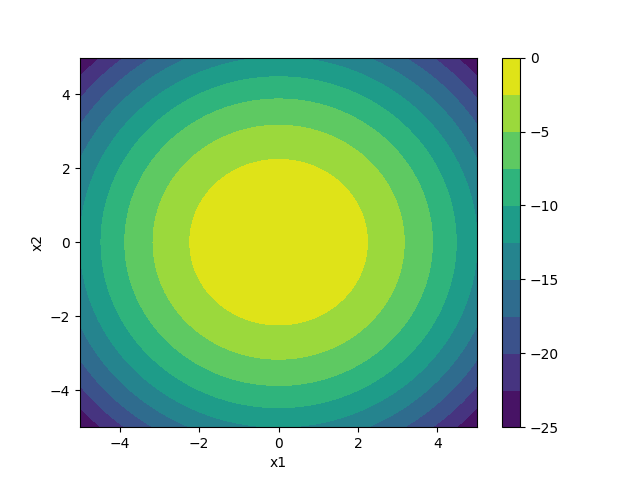}
  \centerline{\footnotesize{(a)}}\medskip
\end{minipage}
\hfill
\begin{minipage}[b]{.49\linewidth}
  \centering
  \includegraphics[width=.99\textwidth]{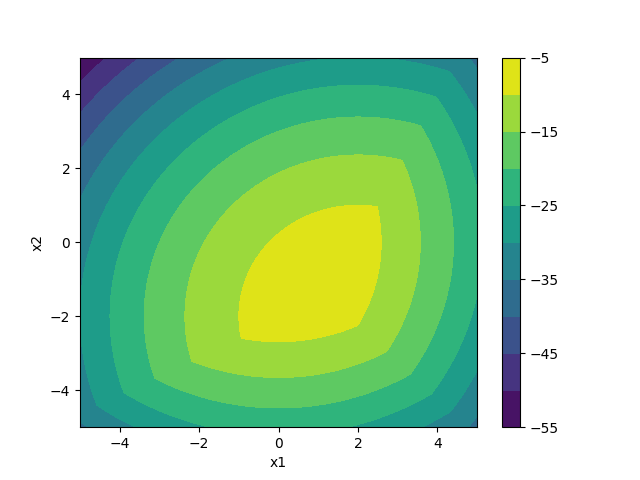}
  \centerline{\footnotesize{(b)}}\medskip
\end{minipage}
%%%% second row
\begin{minipage}[b]{.49\linewidth}
  \centering
    \includegraphics[width=.99\textwidth]{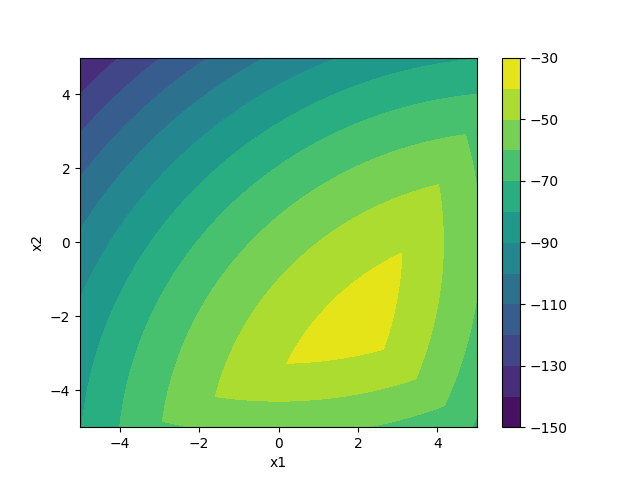}
  \centerline{\footnotesize{(c)}}\medskip
\end{minipage}
\hfill
\begin{minipage}[b]{.49\linewidth}
  \centering
    \includegraphics[width=.99\textwidth]{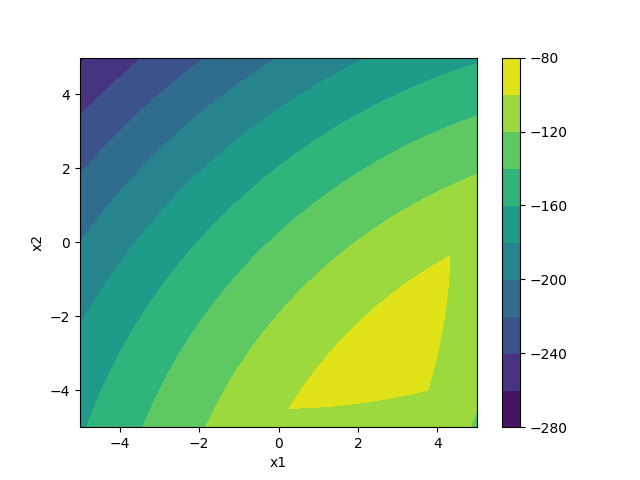}
  \centerline{\footnotesize{(d)}}\medskip
\end{minipage}
\caption{A two dimensional slice of
the graph of $f_2$ in example \ref{eg: f2_hd}. In each figure, the x and y axes correspond to the variables $x_1$ and $x_2$, which are the first and second coordinates of the variable $\bx = (x_1, x_2, 0,\dots, 0)$.
The color is given by the function value $f_2(\bx, t)$.
The figures (a), (b), (c) and (d) show contour lines of the solution $f_2(\bx, t)$ at $t=0$, $t=1$, $t=3$ and $t=5$, respectively. 
\label{fig: egf2_hd}}
\end{figure}

Here, we present a high dimensional example. We choose the dimension to be $n=10$. We consider the solution $f_2: \R^{10}\times [0,+\infty) \to \R$ represented by the neural network in Fig.~\ref{fig: nn2} with three neurons, i.e., we set $m=3$. Similar to the one dimensional case, the activation function $\tilde{J}$ is chosen to be $\tilde{J}(\bx) = -\frac{\|\bx\|_2^2}{2}$ for every $\bx \in \R^{10}$. Hence, by Thm.~\ref{thm:2}, the initial data in the corresponding HJ PDE is given by $J(\bx) = \tilde{J}(\bx) = - \frac{\|\bx\|_2^2}{2}$. The parameters are chosen to be $\bv_1 = (-2,0,0,0,\dots,0)$, $\bv_2 = (2,-2,-1, 0,\dots,0)$, $\bv_3 = (0,2,0,0,\dots,0)$, $b_1 = 0.5$, $b_2 = -5$ and $b_3 = 1$.
Then the Hamiltonian is the corresponding convex piecewise affine function defined in \eqref{eqt:thm2_H_def}.
 
The solution $f_2$ is shown in Fig.~\ref{fig: egf2_hd}. We fix $\bx = (x_1, x_2, 0, \dots, 0)$ and compute $f_2(\bx, t)$ for $t=0$, $1$, $3$ and $5$. In each figure, the color is given by the function value $f_2(\bx,t)$ and the x and y axes represent the variables $x_1$ and $x_2$, respectively. 
The solutions at $t=0$, $t=1$, $t=3$ and $t=5$ are shown in (a), (b), (c) and (d), respectively.
Again, we observe kinks on the level curves in Fig.~\ref{fig: egf2_hd}~(b-d). Therefore, the proposed neural network architecture computes the viscosity solution without numerical smoothing effect. 
\end{example}

\begin{example} \label{eg:f2_Hnorm}
In this example, we consider two HJ PDEs defined for $\bx\in \R^5$, i.e., the dimension is $n=5$.
The initial data $J$ is given by $J(\bx) = -\frac{\|\bx\|_2^2}{2}$ for each $\bx\in\R^5$ and the Hamiltonian $H$ is the $l^1$-norm or the $l^\infty$-norm. The corresponding solutions $f_2$ are shown in Figs.~\ref{fig: egf2_hd_1norm} and \ref{fig: egf2_hd_infnorm}. Similarly as in example \ref{eg: f2_hd}, we consider the variable $\bx = (x_1, x_2, 0, 0,0)$ and show the 2-dimensional slice in each figure. The solutions at $t=0$, $t=1$, $t=3$ and $t=5$ are shown in (a), (b), (c) and (d), respectively, in each figure.

When $H$ is the $l^1$-norm, i.e., $H(\bp) = \|\bp\|_1$ for each $\bp\in\R^5$, the Hamiltonian $H$ can be written in the form of Eq.~\eqref{eqt:thm2_H_def} with $m = 2^n$, $b_i = 0$ for each $i\in \{1,\dots, m\}$ and 
\begin{equation*}
    \{\bv_i\}_{i=1}^m = \{(w_1, w_2,\dots, w_n)\in\R^n\colon w_j \in \{\pm 1\} \, \forall j\in\{1,\dots, n\}\}.
\end{equation*}
The corresponding function $f_2$ is shown in Fig.~\ref{fig: egf2_hd_1norm}.

When $H$ is the $l^\infty$-norm, i.e., $H(\bp) = \|\bp\|_\infty$ for each $\bp\in\R^5$, the Hamiltonian $H$ can be written in the form of Eq.~\eqref{eqt:thm2_H_def} with $m = 2n$, $b_i = 0$ for each $i\in \{1,\dots, m\}$ and 
\begin{equation*}
    \{\bv_i\}_{i=1}^m = \{\pm \boldsymbol{e}_j\}_{j=1}^n,
\end{equation*}
where $\boldsymbol{e}_j$ is the $j-$th coordinate basis vector in $\R^n$.
The corresponding function $f_2$ is shown in Fig.~\ref{fig: egf2_hd_infnorm}.

We observe kinks on the level curves in Fig.~\ref{fig: egf2_hd_1norm}~(b-d) and Fig.~\ref{fig: egf2_hd_infnorm}~(b-d).
These numerical examples show that the proposed neural network architecture computes the viscosity solution to the HJ PDEs without any numerical smoothing effect.

\begin{figure}[ht]
\begin{minipage}[b]{.49\linewidth}
  \centering
  \includegraphics[width=.99\textwidth]{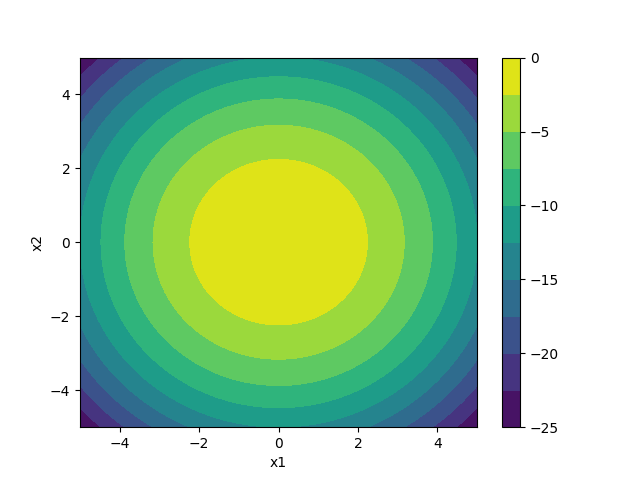}
  \centerline{\footnotesize{(a)}}\medskip
\end{minipage}
\hfill
\begin{minipage}[b]{.49\linewidth}
  \centering
  \includegraphics[width=.99\textwidth]{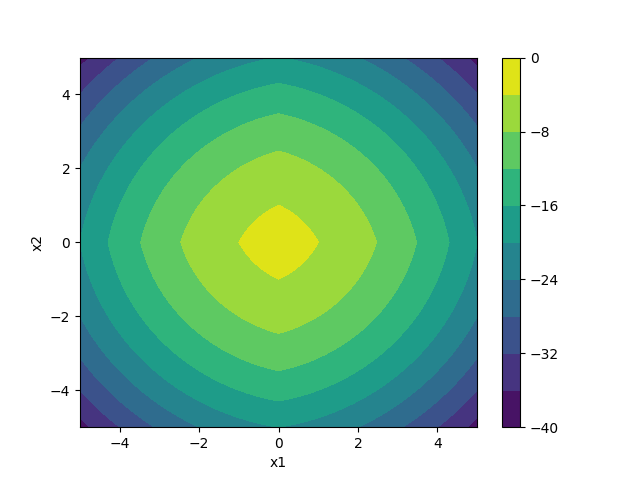}
  \centerline{\footnotesize{(b)}}\medskip
\end{minipage}
%%%% second row
\begin{minipage}[b]{.49\linewidth}
  \centering
    \includegraphics[width=.99\textwidth]{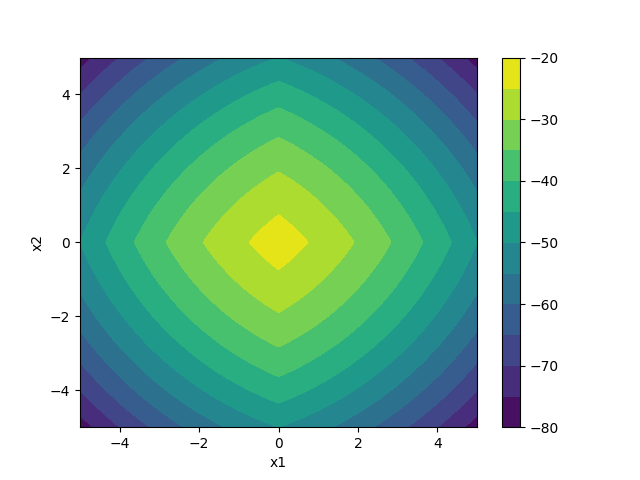}
  \centerline{\footnotesize{(c)}}\medskip
\end{minipage}
\hfill
\begin{minipage}[b]{.49\linewidth}
  \centering
    \includegraphics[width=.99\textwidth]{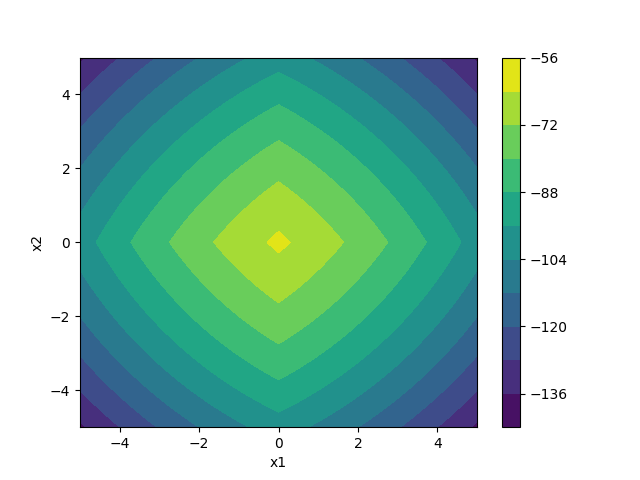}
  \centerline{\footnotesize{(d)}}\medskip
\end{minipage}
\caption{A two dimensional slice of
the graph of $f_2$ in example \ref{eg:f2_Hnorm}. The initial data $J$ is given by $J(\bx) = -\frac{\|\bx\|_2^2}{2}$ and the Hamiltonian $H$ is the $l^1$ norm. 
In each figure, the x and y axes correspond to the variables $x_1$ and $x_2$, which are the first and second coordinates of the variable $\bx = (x_1, x_2, 0,\dots, 0)$.
The color is given by the function value $f_2(\bx, t)$.
The figures (a), (b), (c) and (d) show contour lines of the solution $f_2(\bx, t)$ at $t=0$, $t=1$, $t=3$ and $t=5$, respectively. 
\label{fig: egf2_hd_1norm}}
\end{figure}

\begin{figure}[ht]
\begin{minipage}[b]{.49\linewidth}
  \centering
  \includegraphics[width=.99\textwidth]{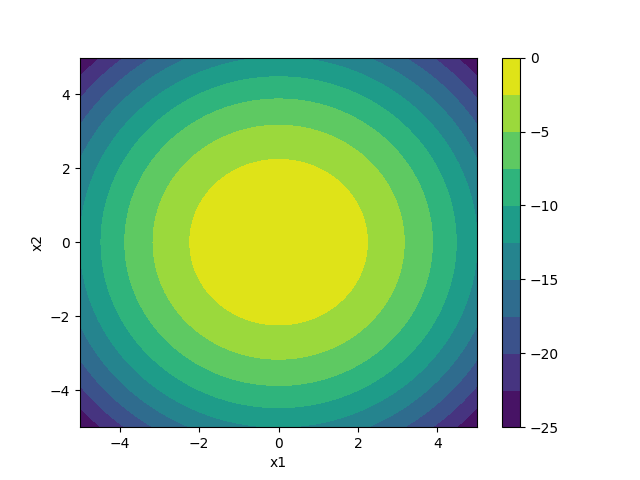}
  \centerline{\footnotesize{(a)}}\medskip
\end{minipage}
\hfill  
\begin{minipage}[b]{.49\linewidth}
  \centering
  \includegraphics[width=.99\textwidth]{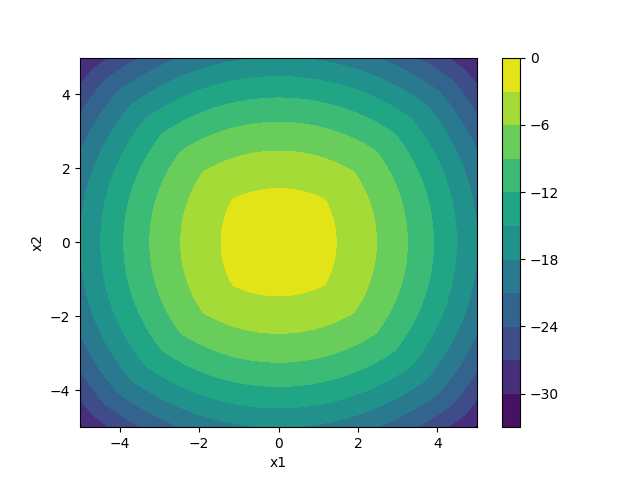}
  \centerline{\footnotesize{(b)}}\medskip
\end{minipage}
%%%% second row
\begin{minipage}[b]{.49\linewidth}
  \centering
    \includegraphics[width=.99\textwidth]{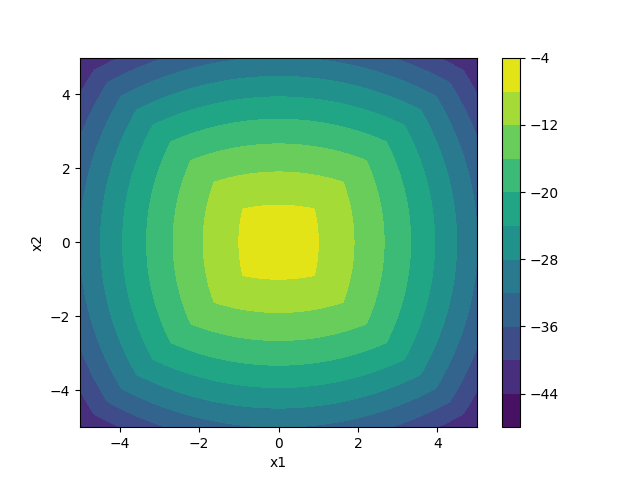}
  \centerline{\footnotesize{(c)}}\medskip
\end{minipage}
\hfill
\begin{minipage}[b]{.49\linewidth}
  \centering
    \includegraphics[width=.99\textwidth]{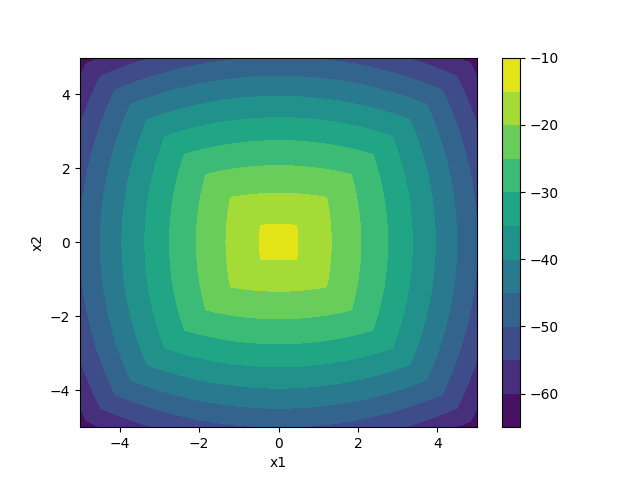}
  \centerline{\footnotesize{(d)}}\medskip
\end{minipage}
\caption{A two dimensional slice of
the graph of $f_2$ in example \ref{eg:f2_Hnorm}. The initial data $J$ is given by $J(\bx) = -\frac{\|\bx\|_2^2}{2}$ and the Hamiltonian $H$ is the $l^\infty$ norm. 
In each figure, the x and y axes correspond to the variables $x_1$ and $x_2$, which are the first and second coordinates of the variable $\bx = (x_1, x_2, 0,\dots, 0)$.
The color is given by the function value $f_2(\bx, t)$.
The figures (a), (b), (c) and (d) show contour lines of the solution $f_2(\bx, t)$ at $t=0$, $t=1$, $t=3$ and $t=5$, respectively. 
\label{fig: egf2_hd_infnorm}}
\end{figure}
\end{example}

\section{Conclusion} \label{sec:conclusion}
In this paper, we investigated two neural network architectures shown in Figs.~\ref{fig: nn1} and \ref{fig: nn2}, and proved that these two architectures represent viscosity solutions to two sets of HJ PDEs whose convex Hamiltonian $\HHJ$ and initial data $\JHJ$ satisfy certain assumptions in Thms.~\ref{thm:1} and \ref{thm:2}, respectively. 
This connection provides a possible interpretation for some neural network architectures. 
Our results suggest that efficient dedicated hardware implementation for neural networks can be leveraged to compute  viscosity solutions of certain HJ PDEs. A future direction consists of implementing these neural networks on FPGA using Xilinx tools (e.g., Xilinx Vitis High Level Synthesis) to evaluate the performance of these FPGA implementations.

In this paper, we only consider the HJ PDEs whose Hamiltonian $H$ does not depend on the state variable $\bx$ and the time variable $t$. Out results suggest further research directions:  what kind of neural network architectures can be used to represent the viscosity solution to certain HJ PDEs whose Hamiltonian depends on $\bx$ or $t$? Note that a generalized Hopf-Lax formula for certain HJ PDEs with state dependent Hamiltonians is proposed in~\cite{Dragoni2007Metric}. However, this formula involves a distance function which is a solution to the Eikonal equation. Hence, it is not straightforward to design a neural network architecture using this representation formula. We propose to investigate novel representation formulas for these HJ PDEs that can be represented using neural networks.

\section*{Acknowledgments}
This research is supported by NSF DMS 1820821 and AFOSR MURI FA9550-20-1-0358. 

\bibliographystyle{model1-num-names}
\bibliography{biblist}

\begin{thebibliography}{118}
\expandafter\ifx\csname natexlab\endcsname\relax\def\natexlab#1{#1}\fi
\providecommand{\url}[1]{\texttt{#1}}
\providecommand{\href}[2]{#2}
\providecommand{\path}[1]{#1}
\providecommand{\DOIprefix}{doi:}
\providecommand{\ArXivprefix}{arXiv:}
\providecommand{\URLprefix}{URL: }
\providecommand{\Pubmedprefix}{pmid:}
\providecommand{\doi}[1]{\href{http://dx.doi.org/#1}{\path{#1}}}
\providecommand{\Pubmed}[1]{\href{pmid:#1}{\path{#1}}}
\providecommand{\bibinfo}[2]{#2}
\ifx\xfnm\relax \def\xfnm[#1]{\unskip,\space#1}\fi
%Type = Book
\bibitem[{Arnol'd(1989)}]{Arnold1989Math}
\bibinfo{author}{V.~I. Arnol'd}, \bibinfo{title}{Mathematical methods of
  classical mechanics}, volume~\bibinfo{volume}{60} of
  \textit{\bibinfo{series}{Graduate Texts in Mathematics}},
  \bibinfo{publisher}{Springer-Verlag, New York}, \bibinfo{year}{1989}.
  \bibinfo{note}{Translated from the 1974 Russian original by K. Vogtmann and
  A. Weinstein, Corrected reprint of the second (1989) edition}.
%Type = Book
\bibitem[{Carath\'{e}odory(1965)}]{Caratheodory1965CalculusI}
\bibinfo{author}{C.~Carath\'{e}odory}, \bibinfo{title}{Calculus of variations
  and partial differential equations of the first order. {P}art {I}: {P}artial
  differential equations of the first order}, Translated by Robert B. Dean and
  Julius J. Brandstatter, \bibinfo{publisher}{Holden-Day, Inc., San
  Francisco-London-Amsterdam}, \bibinfo{year}{1965}.
%Type = Book
\bibitem[{Carath\'{e}odory(1967)}]{Caratheodory1967CalculusII}
\bibinfo{author}{C.~Carath\'{e}odory}, \bibinfo{title}{Calculus of variations
  and partial differential equations of the first order. {P}art {II}:
  {C}alculus of variations}, Translated from the German by Robert B. Dean,
  Julius J. Brandstatter, translating editor, \bibinfo{publisher}{Holden-Day,
  Inc., San Francisco-London-Amsterdam}, \bibinfo{year}{1967}.
%Type = Book
\bibitem[{Courant and Hilbert(1989)}]{Courant1989Methods}
\bibinfo{author}{R.~Courant}, \bibinfo{author}{D.~Hilbert},
  \bibinfo{title}{Methods of mathematical physics. {V}ol. {II}}, Wiley Classics
  Library, \bibinfo{publisher}{John Wiley \& Sons, Inc., New York},
  \bibinfo{year}{1989}. \bibinfo{note}{Partial differential equations, Reprint
  of the 1962 original, A Wiley-Interscience Publication}.
%Type = Book
\bibitem[{Landau and Lifschic(1978)}]{landau1978course}
\bibinfo{author}{L.~Landau}, \bibinfo{author}{E.~Lifschic},
  \bibinfo{title}{Course of theoretical physics. vol. 1: Mechanics},
  \bibinfo{publisher}{Oxford}, \bibinfo{year}{1978}.
%Type = Book
\bibitem[{Bardi and Capuzzo-Dolcetta(1997)}]{Bardi1997Optimal}
\bibinfo{author}{M.~Bardi}, \bibinfo{author}{I.~Capuzzo-Dolcetta},
  \bibinfo{title}{Optimal control and viscosity solutions of
  {H}amilton-{J}acobi-{B}ellman equations}, Systems \& Control: Foundations \&
  Applications, \bibinfo{publisher}{Birkh\"{a}user Boston, Inc., Boston, MA},
  \bibinfo{year}{1997}. \DOIprefix\doi{10.1007/978-0-8176-4755-1},
  \bibinfo{note}{with appendices by Maurizio Falcone and Pierpaolo Soravia}.
%Type = Book
\bibitem[{Elliott(1987)}]{Elliott1987Viscosity}
\bibinfo{author}{R.~J. Elliott}, \bibinfo{title}{Viscosity solutions and
  optimal control}, volume \bibinfo{volume}{165} of
  \textit{\bibinfo{series}{Pitman Research Notes in Mathematics Series}},
  \bibinfo{publisher}{Longman Scientific \& Technical, Harlow; John Wiley \&
  Sons, Inc., New York}, \bibinfo{year}{1987}.
%Type = Article
\bibitem[{Fleming and Rishel(1976)}]{fleming1976deterministic}
\bibinfo{author}{W.~H. Fleming}, \bibinfo{author}{R.~W. Rishel},
\newblock \bibinfo{title}{Deterministic and stochastic optimal control},
\newblock \bibinfo{journal}{Bulletin of the American Mathematical Society}
  \bibinfo{volume}{82} (\bibinfo{year}{1976}) \bibinfo{pages}{869--870}.
%Type = Book
\bibitem[{Fleming and Soner(2006)}]{fleming2006controlled}
\bibinfo{author}{W.~H. Fleming}, \bibinfo{author}{H.~M. Soner},
  \bibinfo{title}{Controlled Markov processes and viscosity solutions},
  volume~\bibinfo{volume}{25}, \bibinfo{publisher}{Springer Science \& Business
  Media}, \bibinfo{year}{2006}.
%Type = Book
\bibitem[{McEneaney(2006)}]{mceneaney2006max}
\bibinfo{author}{W.~McEneaney}, \bibinfo{title}{Max-plus methods for nonlinear
  control and estimation}, \bibinfo{publisher}{Springer Science \& Business
  Media}, \bibinfo{year}{2006}.
%Type = Article
\bibitem[{Barron et~al.(1984)Barron, Evans, and Jensen}]{BARRON1984213}
\bibinfo{author}{E.~Barron}, \bibinfo{author}{L.~Evans},
  \bibinfo{author}{R.~Jensen},
\newblock \bibinfo{title}{Viscosity solutions of {I}saacs' equations and
  differential games with {L}ipschitz controls},
\newblock \bibinfo{journal}{Journal of Differential Equations}
  \bibinfo{volume}{53} (\bibinfo{year}{1984}) \bibinfo{pages}{213 -- 233}.
%Type = Article
\bibitem[{Buckdahn et~al.(2011)Buckdahn, Cardaliaguet, and
  Quincampoix}]{Buckdahn2011Recent}
\bibinfo{author}{R.~Buckdahn}, \bibinfo{author}{P.~Cardaliaguet},
  \bibinfo{author}{M.~Quincampoix},
\newblock \bibinfo{title}{Some recent aspects of differential game theory},
\newblock \bibinfo{journal}{Dynamic Games and Applications} \bibinfo{volume}{1}
  (\bibinfo{year}{2011}) \bibinfo{pages}{74--114}.
%Type = Article
\bibitem[{Evans and Souganidis(1984)}]{Evans1984Differential}
\bibinfo{author}{L.~C. Evans}, \bibinfo{author}{P.~E. Souganidis},
\newblock \bibinfo{title}{Differential games and representation formulas for
  solutions of {H}amilton-{J}acobi-{I}saacs equations},
\newblock \bibinfo{journal}{Indiana University Mathematics Journal}
  \bibinfo{volume}{33} (\bibinfo{year}{1984}) \bibinfo{pages}{773--797}.
%Type = Article
\bibitem[{Ishii(1988)}]{Ishii1988Representation}
\bibinfo{author}{H.~Ishii},
\newblock \bibinfo{title}{Representation of solutions of {H}amilton-{J}acobi
  equations},
\newblock \bibinfo{journal}{Nonlinear Analysis: Theory, Methods \&
  Applications} \bibinfo{volume}{12} (\bibinfo{year}{1988}) \bibinfo{pages}{121
  -- 146}.
%Type = Article
\bibitem[{Darbon(2015)}]{darbon2015convex}
\bibinfo{author}{J.~Darbon},
\newblock \bibinfo{title}{On convex finite-dimensional variational methods in
  imaging sciences and {H}amilton--{J}acobi equations},
\newblock \bibinfo{journal}{SIAM Journal on Imaging Sciences}
  \bibinfo{volume}{8} (\bibinfo{year}{2015}) \bibinfo{pages}{2268--2293}.
%Type = Article
\bibitem[{Darbon and Meng(2020)}]{darbon2019decomposition}
\bibinfo{author}{J.~Darbon}, \bibinfo{author}{T.~Meng},
\newblock \bibinfo{title}{On decomposition models in imaging sciences and
  multi-time {H}amilton--{J}acobi partial differential equations},
\newblock \bibinfo{journal}{SIAM Journal on Imaging Sciences}
  \bibinfo{volume}{13} (\bibinfo{year}{2020}) \bibinfo{pages}{971--1014}.
%Type = Article
\bibitem[{Darbon and Osher(2016)}]{Darbon2016Algorithms}
\bibinfo{author}{J.~Darbon}, \bibinfo{author}{S.~Osher},
\newblock \bibinfo{title}{Algorithms for overcoming the curse of dimensionality
  for certain {H}amilton--{J}acobi equations arising in control theory and
  elsewhere},
\newblock \bibinfo{journal}{Research in the Mathematical Sciences}
  \bibinfo{volume}{3} (\bibinfo{year}{2016}) \bibinfo{pages}{19}.
%Type = Article
\bibitem[{Darbon et~al.(2020)Darbon, Langlois, and Meng}]{darbon2019overcoming}
\bibinfo{author}{J.~Darbon}, \bibinfo{author}{G.~P. Langlois},
  \bibinfo{author}{T.~Meng},
\newblock \bibinfo{title}{Overcoming the curse of dimensionality for some
  {H}amilton-{J}acobi partial differential equations via neural network
  architectures},
\newblock \bibinfo{journal}{Res. Math. Sci.} \bibinfo{volume}{7}
  (\bibinfo{year}{2020}) \bibinfo{pages}{20}.
%Type = Article
\bibitem[{Bardi and Evans(1984)}]{bardi1984hopf}
\bibinfo{author}{M.~Bardi}, \bibinfo{author}{L.~Evans},
\newblock \bibinfo{title}{On {H}opf's formulas for solutions of
  {H}amilton-{J}acobi equations},
\newblock \bibinfo{journal}{Nonlinear Analysis: Theory, Methods \&
  Applications} \bibinfo{volume}{8} (\bibinfo{year}{1984}) \bibinfo{pages}{1373
  -- 1381}.
%Type = Book
\bibitem[{Barles(1994)}]{barles1994solutions}
\bibinfo{author}{G.~Barles}, \bibinfo{title}{Solutions de viscosit{\'e} des
  {\'e}quations de Hamilton-Jacobi}, Math{\'e}matiques et Applications,
  \bibinfo{publisher}{Springer-Verlag Berlin Heidelberg}, \bibinfo{year}{1994}.
%Type = Article
\bibitem[{Crandall et~al.(1992)Crandall, Ishii, and Lions}]{crandall1992user}
\bibinfo{author}{M.~G. Crandall}, \bibinfo{author}{H.~Ishii},
  \bibinfo{author}{P.-L. Lions},
\newblock \bibinfo{title}{User's guide to viscosity solutions of second order
  partial differential equations},
\newblock \bibinfo{journal}{Bulletin of the American mathematical society}
  \bibinfo{volume}{27} (\bibinfo{year}{1992}) \bibinfo{pages}{1--67}.
%Type = Inproceedings
\bibitem[{{LeCun}(2019)}]{lecun2019isscc}
\bibinfo{author}{Y.~{LeCun}},
\newblock \bibinfo{title}{1.1 deep learning hardware: Past, present, and
  future},
\newblock in: \bibinfo{booktitle}{2019 IEEE International Solid- State Circuits
  Conference - (ISSCC)}, \bibinfo{year}{2019}, pp. \bibinfo{pages}{12--19}.
  \DOIprefix\doi{10.1109/ISSCC.2019.8662396}.
%Type = Inproceedings
\bibitem[{Farabet et~al.(2011)Farabet, LeCun, Kavukcuoglu, Culurciello,
  Martini, Akselrod, and Talay}]{farabet-suml-11}
\bibinfo{author}{C.~Farabet}, \bibinfo{author}{Y.~LeCun},
  \bibinfo{author}{K.~Kavukcuoglu}, \bibinfo{author}{E.~Culurciello},
  \bibinfo{author}{B.~Martini}, \bibinfo{author}{P.~Akselrod},
  \bibinfo{author}{S.~Talay},
\newblock \bibinfo{title}{Large-scale fpga-based convolutional networks},
\newblock in: \bibinfo{editor}{R.~Bekkerman}, \bibinfo{editor}{M.~Bilenko},
  \bibinfo{editor}{J.~Langford} (Eds.), \bibinfo{booktitle}{Scaling up Machine
  Learning: Parallel and Distributed Approaches}, \bibinfo{publisher}{Cambridge
  University Press}, \bibinfo{year}{2011}.
%Type = Inproceedings
\bibitem[{Farabet et~al.(2009{\natexlab{a}})Farabet, poulet, Han, and
  LeCun}]{farabet-fpl-09}
\bibinfo{author}{C.~Farabet}, \bibinfo{author}{C.~poulet},
  \bibinfo{author}{J.~Han}, \bibinfo{author}{Y.~LeCun},
\newblock \bibinfo{title}{Cnp: An fpga-based processor for convolutional
  networks},
\newblock in: \bibinfo{booktitle}{International Conference on Field
  Programmable Logic and Applications}, \bibinfo{publisher}{IEEE},
  \bibinfo{address}{Prague}, \bibinfo{year}{2009}{\natexlab{a}}.
%Type = Inproceedings
\bibitem[{Farabet et~al.(2009{\natexlab{b}})Farabet, Poulet, and
  LeCun}]{farabet.09.iccvw}
\bibinfo{author}{C.~Farabet}, \bibinfo{author}{C.~Poulet},
  \bibinfo{author}{Y.~LeCun},
\newblock \bibinfo{title}{An fpga-based stream processor for embedded real-time
  vision with convolutional networks},
\newblock in: \bibinfo{booktitle}{2009 IEEE 12th International Conference on
  Computer Vision Workshops, ICCV Workshops}, \bibinfo{publisher}{IEEE Computer
  Society}, \bibinfo{address}{Los Alamitos, CA, USA},
  \bibinfo{year}{2009}{\natexlab{b}}, pp. \bibinfo{pages}{878--885}. \URLprefix
  \url{https://doi.ieeecomputersociety.org/10.1109/ICCVW.2009.5457611}.
  \DOIprefix\doi{10.1109/ICCVW.2009.5457611}.
%Type = Article
\bibitem[{Banerjee et~al.(2019)Banerjee, Georganas, Kalamkar, Ziv, Segal,
  Anderson, and Heinecke}]{banerjeeEtal2019sfi}
\bibinfo{author}{K.~Banerjee}, \bibinfo{author}{E.~Georganas},
  \bibinfo{author}{D.~Kalamkar}, \bibinfo{author}{B.~Ziv},
  \bibinfo{author}{E.~Segal}, \bibinfo{author}{C.~Anderson},
  \bibinfo{author}{A.~Heinecke},
\newblock \bibinfo{title}{Optimizing deep learning rnn topologies on intel
  architecture},
\newblock \bibinfo{journal}{Supercomputing Frontiers and Innovations}
  \bibinfo{volume}{6} (\bibinfo{year}{2019}).
%Type = Inproceedings
\bibitem[{Jouppi et~al.(2017)Jouppi, Young, Patil, Patterson, Agrawal, Bajwa,
  Bates, Bhatia, Boden, Borchers, and et~al.}]{googleTPU17}
\bibinfo{author}{N.~P. Jouppi}, \bibinfo{author}{C.~Young},
  \bibinfo{author}{N.~Patil}, \bibinfo{author}{D.~Patterson},
  \bibinfo{author}{G.~Agrawal}, \bibinfo{author}{R.~Bajwa},
  \bibinfo{author}{S.~Bates}, \bibinfo{author}{S.~Bhatia},
  \bibinfo{author}{N.~Boden}, \bibinfo{author}{A.~Borchers},
  \bibinfo{author}{et~al.},
\newblock \bibinfo{title}{In-datacenter performance analysis of a tensor
  processing unit},
\newblock in: \bibinfo{booktitle}{Proceedings of the 44th Annual International
  Symposium on Computer Architecture}, ISCA ’17,
  \bibinfo{publisher}{Association for Computing Machinery},
  \bibinfo{address}{New York, NY, USA}, \bibinfo{year}{2017}, p.
  \bibinfo{pages}{1–12}. \URLprefix
  \url{https://doi.org/10.1145/3079856.3080246}.
  \DOIprefix\doi{10.1145/3079856.3080246}.
%Type = Article
\bibitem[{Kundu et~al.(2019)Kundu, Srinivasan, Qin, Kalamkar, Mellempudi, Das,
  Banerjee, Kaul, and Dubey}]{kundu2019ktanh}
\bibinfo{author}{A.~Kundu}, \bibinfo{author}{S.~Srinivasan},
  \bibinfo{author}{E.~C. Qin}, \bibinfo{author}{D.~Kalamkar},
  \bibinfo{author}{N.~K. Mellempudi}, \bibinfo{author}{D.~Das},
  \bibinfo{author}{K.~Banerjee}, \bibinfo{author}{B.~Kaul},
  \bibinfo{author}{P.~Dubey},
\newblock \bibinfo{title}{K-tanh: Hardware efficient activations for deep
  learning},
\newblock \bibinfo{journal}{arXiv preprint arXiv:1909.07729}
  (\bibinfo{year}{2019}).
%Type = Article
\bibitem[{Chen et~al.(2020)Chen, van Gelder, van~de Ven, Amitonov, de~Wilde,
  Euler, Broersma, Bobbert, Zwanenburg, and van~der
  Wiel}]{chen2020classification}
\bibinfo{author}{T.~Chen}, \bibinfo{author}{J.~van Gelder},
  \bibinfo{author}{B.~van~de Ven}, \bibinfo{author}{S.~V. Amitonov},
  \bibinfo{author}{B.~de~Wilde}, \bibinfo{author}{H.-C.~R. Euler},
  \bibinfo{author}{H.~Broersma}, \bibinfo{author}{P.~A. Bobbert},
  \bibinfo{author}{F.~A. Zwanenburg}, \bibinfo{author}{W.~G. van~der Wiel},
\newblock \bibinfo{title}{Classification with a disordered dopant-atom network
  in silicon},
\newblock \bibinfo{journal}{Nature} \bibinfo{volume}{577}
  (\bibinfo{year}{2020}) \bibinfo{pages}{341--345}.
%Type = Article
\bibitem[{Hirjibehedin(2020)}]{Hirjibehedin.20.nature}
\bibinfo{author}{C.~Hirjibehedin},
\newblock \bibinfo{title}{Evolution of circuits for machine learning},
\newblock \bibinfo{journal}{Nature} \bibinfo{volume}{577}
  (\bibinfo{year}{2020}) \bibinfo{pages}{320--321}.
%Type = Article
\bibitem[{Akian et~al.(2006)Akian, Bapat, and Gaubert}]{akian2006max}
\bibinfo{author}{M.~Akian}, \bibinfo{author}{R.~Bapat},
  \bibinfo{author}{S.~Gaubert},
\newblock \bibinfo{title}{Max-plus algebra},
\newblock \bibinfo{journal}{Handbook of linear algebra} \bibinfo{volume}{39}
  (\bibinfo{year}{2006}).
%Type = Article
\bibitem[{Akian et~al.(2008)Akian, Gaubert, and Lakhoua}]{akian2008max}
\bibinfo{author}{M.~Akian}, \bibinfo{author}{S.~Gaubert},
  \bibinfo{author}{A.~Lakhoua},
\newblock \bibinfo{title}{The max-plus finite element method for solving
  deterministic optimal control problems: basic properties and convergence
  analysis},
\newblock \bibinfo{journal}{SIAM Journal on Control and Optimization}
  \bibinfo{volume}{47} (\bibinfo{year}{2008}) \bibinfo{pages}{817--848}.
%Type = Inproceedings
\bibitem[{Dower et~al.(2015)Dower, McEneaney, and Zhang}]{dower2015max}
\bibinfo{author}{P.~M. Dower}, \bibinfo{author}{W.~M. McEneaney},
  \bibinfo{author}{H.~Zhang},
\newblock \bibinfo{title}{Max-plus fundamental solution semigroups for optimal
  control problems},
\newblock in: \bibinfo{booktitle}{2015 Proceedings of the Conference on Control
  and its Applications}, \bibinfo{organization}{SIAM}, \bibinfo{year}{2015},
  pp. \bibinfo{pages}{368--375}.
%Type = Article
\bibitem[{Fleming and McEneaney(2000)}]{Fleming2000Max}
\bibinfo{author}{W.~Fleming}, \bibinfo{author}{W.~McEneaney},
\newblock \bibinfo{title}{A max-plus-based algorithm for a
  {H}amilton--{J}acobi--{B}ellman equation of nonlinear filtering},
\newblock \bibinfo{journal}{SIAM Journal on Control and Optimization}
  \bibinfo{volume}{38} (\bibinfo{year}{2000}) \bibinfo{pages}{683--710}.
%Type = Inproceedings
\bibitem[{Gaubert et~al.(2011)Gaubert, McEneaney, and Qu}]{gaubert2011curse}
\bibinfo{author}{S.~Gaubert}, \bibinfo{author}{W.~McEneaney},
  \bibinfo{author}{Z.~Qu},
\newblock \bibinfo{title}{Curse of dimensionality reduction in max-plus based
  approximation methods: Theoretical estimates and improved pruning
  algorithms},
\newblock in: \bibinfo{booktitle}{2011 50th IEEE Conference on Decision and
  Control and European Control Conference}, \bibinfo{organization}{IEEE},
  \bibinfo{year}{2011}, pp. \bibinfo{pages}{1054--1061}.
%Type = Article
\bibitem[{McEneaney(2007)}]{McEneaney2007COD}
\bibinfo{author}{W.~McEneaney},
\newblock \bibinfo{title}{A curse-of-dimensionality-free numerical method for
  solution of certain {HJB} {PDE}s},
\newblock \bibinfo{journal}{SIAM Journal on Control and Optimization}
  \bibinfo{volume}{46} (\bibinfo{year}{2007}) \bibinfo{pages}{1239--1276}.
%Type = Inproceedings
\bibitem[{McEneaney et~al.(2008)McEneaney, Deshpande, and
  Gaubert}]{mceneaney2008curse}
\bibinfo{author}{W.~M. McEneaney}, \bibinfo{author}{A.~Deshpande},
  \bibinfo{author}{S.~Gaubert},
\newblock \bibinfo{title}{Curse-of-complexity attenuation in the
  curse-of-dimensionality-free method for {HJB} {PDE}s},
\newblock in: \bibinfo{booktitle}{2008 American Control Conference},
  \bibinfo{organization}{IEEE}, \bibinfo{year}{2008}, pp.
  \bibinfo{pages}{4684--4690}.
%Type = Article
\bibitem[{McEneaney and Kluberg(2009)}]{mceneaney2009convergence}
\bibinfo{author}{W.~M. McEneaney}, \bibinfo{author}{L.~J. Kluberg},
\newblock \bibinfo{title}{Convergence rate for a curse-of-dimensionality-free
  method for a class of {HJB} {PDE}s},
\newblock \bibinfo{journal}{SIAM Journal on Control and Optimization}
  \bibinfo{volume}{48} (\bibinfo{year}{2009}) \bibinfo{pages}{3052--3079}.
%Type = Article
\bibitem[{Alla et~al.(2019)Alla, Falcone, and Saluzzi}]{alla2019efficient}
\bibinfo{author}{A.~Alla}, \bibinfo{author}{M.~Falcone},
  \bibinfo{author}{L.~Saluzzi},
\newblock \bibinfo{title}{An efficient {DP} algorithm on a tree-structure for
  finite horizon optimal control problems},
\newblock \bibinfo{journal}{SIAM Journal on Scientific Computing}
  \bibinfo{volume}{41} (\bibinfo{year}{2019}) \bibinfo{pages}{A2384--A2406}.
%Type = Article
\bibitem[{Bertsekas(2019)}]{bertsekas2019reinforcement}
\bibinfo{author}{D.~P. Bertsekas},
\newblock \bibinfo{title}{Reinforcement learning and optimal control},
\newblock \bibinfo{journal}{Athena Scientific, Belmont, Massachusetts}
  (\bibinfo{year}{2019}).
%Type = Article
\bibitem[{Dolgov et~al.(2019)Dolgov, Kalise, and Kunisch}]{dolgov2019tensor}
\bibinfo{author}{S.~Dolgov}, \bibinfo{author}{D.~Kalise},
  \bibinfo{author}{K.~Kunisch},
\newblock \bibinfo{title}{A tensor decomposition approach for high-dimensional
  {H}amilton-{J}acobi-{B}ellman equations},
\newblock \bibinfo{journal}{arXiv preprint arXiv:1908.01533}
  (\bibinfo{year}{2019}).
%Type = Inproceedings
\bibitem[{Horowitz et~al.(2014)Horowitz, Damle, and
  Burdick}]{horowitz2014linear}
\bibinfo{author}{M.~B. Horowitz}, \bibinfo{author}{A.~Damle},
  \bibinfo{author}{J.~W. Burdick},
\newblock \bibinfo{title}{Linear {H}amilton {J}acobi {B}ellman equations in
  high dimensions},
\newblock in: \bibinfo{booktitle}{53rd IEEE Conference on Decision and
  Control}, \bibinfo{organization}{IEEE}, \bibinfo{year}{2014}, pp.
  \bibinfo{pages}{5880--5887}.
%Type = Article
\bibitem[{Todorov(2009)}]{todorov2009efficient}
\bibinfo{author}{E.~Todorov},
\newblock \bibinfo{title}{Efficient computation of optimal actions},
\newblock \bibinfo{journal}{Proceedings of the national academy of sciences}
  \bibinfo{volume}{106} (\bibinfo{year}{2009}) \bibinfo{pages}{11478--11483}.
%Type = Article
\bibitem[{Bokanowski et~al.(2013)Bokanowski, Garcke, Griebel, and
  Klompmaker}]{bokanowski2013adaptive}
\bibinfo{author}{O.~Bokanowski}, \bibinfo{author}{J.~Garcke},
  \bibinfo{author}{M.~Griebel}, \bibinfo{author}{I.~Klompmaker},
\newblock \bibinfo{title}{An adaptive sparse grid semi-{L}agrangian scheme for
  first order {H}amilton-{J}acobi {B}ellman equations},
\newblock \bibinfo{journal}{Journal of Scientific Computing}
  \bibinfo{volume}{55} (\bibinfo{year}{2013}) \bibinfo{pages}{575--605}.
%Type = Article
\bibitem[{Garcke and Kr{\"o}ner(2017)}]{garcke2017suboptimal}
\bibinfo{author}{J.~Garcke}, \bibinfo{author}{A.~Kr{\"o}ner},
\newblock \bibinfo{title}{Suboptimal feedback control of {PDE}s by solving
  {HJB} equations on adaptive sparse grids},
\newblock \bibinfo{journal}{Journal of Scientific Computing}
  \bibinfo{volume}{70} (\bibinfo{year}{2017}) \bibinfo{pages}{1--28}.
%Type = Article
\bibitem[{Kang and Wilcox(2017)}]{kang2017mitigating}
\bibinfo{author}{W.~Kang}, \bibinfo{author}{L.~C. Wilcox},
\newblock \bibinfo{title}{Mitigating the curse of dimensionality: sparse grid
  characteristics method for optimal feedback control and {HJB} equations},
\newblock \bibinfo{journal}{Computational Optimization and Applications}
  \bibinfo{volume}{68} (\bibinfo{year}{2017}) \bibinfo{pages}{289--315}.
%Type = Article
\bibitem[{Alla et~al.(2017)Alla, Falcone, and Volkwein}]{alla2017error}
\bibinfo{author}{A.~Alla}, \bibinfo{author}{M.~Falcone},
  \bibinfo{author}{S.~Volkwein},
\newblock \bibinfo{title}{Error analysis for {POD} approximations of infinite
  horizon problems via the dynamic programming approach},
\newblock \bibinfo{journal}{SIAM Journal on Control and Optimization}
  \bibinfo{volume}{55} (\bibinfo{year}{2017}) \bibinfo{pages}{3091--3115}.
%Type = Article
\bibitem[{Kunisch et~al.(2004)Kunisch, Volkwein, and Xie}]{kunisch2004hjb}
\bibinfo{author}{K.~Kunisch}, \bibinfo{author}{S.~Volkwein},
  \bibinfo{author}{L.~Xie},
\newblock \bibinfo{title}{{HJB}-{POD}-based feedback design for the optimal
  control of evolution problems},
\newblock \bibinfo{journal}{SIAM Journal on Applied Dynamical Systems}
  \bibinfo{volume}{3} (\bibinfo{year}{2004}) \bibinfo{pages}{701--722}.
%Type = Article
\bibitem[{Kalise et~al.(2019)Kalise, Kundu, and Kunisch}]{kalise2019robust}
\bibinfo{author}{D.~Kalise}, \bibinfo{author}{S.~Kundu},
  \bibinfo{author}{K.~Kunisch},
\newblock \bibinfo{title}{Robust feedback control of nonlinear {PDE}s by
  numerical approximation of high-dimensional {H}amilton-{J}acobi-{I}saacs
  equations},
\newblock \bibinfo{journal}{arXiv preprint arXiv:1905.06276}
  (\bibinfo{year}{2019}).
%Type = Article
\bibitem[{Kalise and Kunisch(2018)}]{kalise2018polynomial}
\bibinfo{author}{D.~Kalise}, \bibinfo{author}{K.~Kunisch},
\newblock \bibinfo{title}{Polynomial approximation of high-dimensional
  {H}amilton--{J}acobi--{B}ellman equations and applications to feedback
  control of semilinear parabolic {PDE}s},
\newblock \bibinfo{journal}{SIAM Journal on Scientific Computing}
  \bibinfo{volume}{40} (\bibinfo{year}{2018}) \bibinfo{pages}{A629--A652}.
%Type = Article
\bibitem[{Yegorov and Dower(2017)}]{yegorov2017perspectives}
\bibinfo{author}{I.~Yegorov}, \bibinfo{author}{P.~M. Dower},
\newblock \bibinfo{title}{Perspectives on characteristics based
  curse-of-dimensionality-free numerical approaches for solving
  {H}amilton--{J}acobi equations},
\newblock \bibinfo{journal}{Applied Mathematics \& Optimization}
  (\bibinfo{year}{2017}) \bibinfo{pages}{1--49}.
%Type = Article
\bibitem[{Bachouch et~al.(2018)Bachouch, Hur{\'e}, Langren{\'e}, and
  Pham}]{bachouch2018deep}
\bibinfo{author}{A.~Bachouch}, \bibinfo{author}{C.~Hur{\'e}},
  \bibinfo{author}{N.~Langren{\'e}}, \bibinfo{author}{H.~Pham},
\newblock \bibinfo{title}{Deep neural networks algorithms for stochastic
  control problems on finite horizon: numerical applications},
\newblock \bibinfo{journal}{arXiv preprint arXiv:1812.05916}
  (\bibinfo{year}{2018}).
%Type = Inproceedings
\bibitem[{{Djeridane} and {Lygeros}(2006)}]{Djeridane2006Neural}
\bibinfo{author}{B.~{Djeridane}}, \bibinfo{author}{J.~{Lygeros}},
\newblock \bibinfo{title}{Neural approximation of {PDE} solutions: An
  application to reachability computations},
\newblock in: \bibinfo{booktitle}{Proceedings of the 45th IEEE Conference on
  Decision and Control}, \bibinfo{year}{2006}, pp. \bibinfo{pages}{3034--3039}.
  \DOIprefix\doi{10.1109/CDC.2006.377184}.
%Type = Article
\bibitem[{Jiang et~al.(2016)Jiang, Chou, Chen, and Tomlin}]{jiang2016using}
\bibinfo{author}{F.~Jiang}, \bibinfo{author}{G.~Chou},
  \bibinfo{author}{M.~Chen}, \bibinfo{author}{C.~J. Tomlin},
\newblock \bibinfo{title}{Using neural networks to compute approximate and
  guaranteed feasible {H}amilton-{J}acobi-{B}ellman {PDE} solutions},
\newblock \bibinfo{journal}{arXiv preprint arXiv:1611.03158}
  (\bibinfo{year}{2016}).
%Type = Article
\bibitem[{Han et~al.(2018)Han, Jentzen, and E}]{Han2018Solving}
\bibinfo{author}{J.~Han}, \bibinfo{author}{A.~Jentzen}, \bibinfo{author}{W.~E},
\newblock \bibinfo{title}{Solving high-dimensional partial differential
  equations using deep learning},
\newblock \bibinfo{journal}{Proceedings of the National Academy of Sciences}
  \bibinfo{volume}{115} (\bibinfo{year}{2018}) \bibinfo{pages}{8505--8510}.
%Type = Article
\bibitem[{Hur{\'e} et~al.(2018)Hur{\'e}, Pham, Bachouch, and
  Langren{\'e}}]{hure2018deep}
\bibinfo{author}{C.~Hur{\'e}}, \bibinfo{author}{H.~Pham},
  \bibinfo{author}{A.~Bachouch}, \bibinfo{author}{N.~Langren{\'e}},
\newblock \bibinfo{title}{Deep neural networks algorithms for stochastic
  control problems on finite horizon, part {I}: convergence analysis},
\newblock \bibinfo{journal}{arXiv preprint arXiv:1812.04300}
  (\bibinfo{year}{2018}).
%Type = Article
\bibitem[{Hur{\'e} et~al.(2019)Hur{\'e}, Pham, and Warin}]{hure2019some}
\bibinfo{author}{C.~Hur{\'e}}, \bibinfo{author}{H.~Pham},
  \bibinfo{author}{X.~Warin},
\newblock \bibinfo{title}{Some machine learning schemes for high-dimensional
  nonlinear {PDE}s},
\newblock \bibinfo{journal}{arXiv preprint arXiv:1902.01599}
  (\bibinfo{year}{2019}).
%Type = Article
\bibitem[{Lambrianides et~al.(2019)Lambrianides, Gong, and
  Venturi}]{lambrianides2019new}
\bibinfo{author}{P.~Lambrianides}, \bibinfo{author}{Q.~Gong},
  \bibinfo{author}{D.~Venturi},
\newblock \bibinfo{title}{A new scalable algorithm for computational optimal
  control under uncertainty},
\newblock \bibinfo{journal}{arXiv preprint arXiv:1909.07960}
  (\bibinfo{year}{2019}).
%Type = Inproceedings
\bibitem[{{Niarchos} and {Lygeros}(2006)}]{Niarchos2006Neural}
\bibinfo{author}{K.~N. {Niarchos}}, \bibinfo{author}{J.~{Lygeros}},
\newblock \bibinfo{title}{A neural approximation to continuous time
  reachability computations},
\newblock in: \bibinfo{booktitle}{Proceedings of the 45th IEEE Conference on
  Decision and Control}, \bibinfo{year}{2006}, pp. \bibinfo{pages}{6313--6318}.
  \DOIprefix\doi{10.1109/CDC.2006.377358}.
%Type = Article
\bibitem[{Reisinger and Zhang(2019)}]{reisinger2019rectified}
\bibinfo{author}{C.~Reisinger}, \bibinfo{author}{Y.~Zhang},
\newblock \bibinfo{title}{Rectified deep neural networks overcome the curse of
  dimensionality for nonsmooth value functions in zero-sum games of nonlinear
  stiff systems},
\newblock \bibinfo{journal}{arXiv preprint arXiv:1903.06652}
  (\bibinfo{year}{2019}).
%Type = Article
\bibitem[{Royo and Tomlin(2016)}]{royo2016recursive}
\bibinfo{author}{V.~R. Royo}, \bibinfo{author}{C.~Tomlin},
\newblock \bibinfo{title}{Recursive regression with neural networks:
  Approximating the {HJI} {PDE} solution},
\newblock \bibinfo{journal}{arXiv preprint arXiv:1611.02739}
  (\bibinfo{year}{2016}).
%Type = Article
\bibitem[{Sirignano and Spiliopoulos(2018)}]{Sirignano2018DGM}
\bibinfo{author}{J.~Sirignano}, \bibinfo{author}{K.~Spiliopoulos},
\newblock \bibinfo{title}{{DGM}: A deep learning algorithm for solving partial
  differential equations},
\newblock \bibinfo{journal}{Journal of Computational Physics}
  \bibinfo{volume}{375} (\bibinfo{year}{2018}) \bibinfo{pages}{1339 -- 1364}.
%Type = Article
\bibitem[{Beck et~al.(2018)Beck, Becker, Grohs, Jaafari, and
  Jentzen}]{beck2018solving}
\bibinfo{author}{C.~Beck}, \bibinfo{author}{S.~Becker},
  \bibinfo{author}{P.~Grohs}, \bibinfo{author}{N.~Jaafari},
  \bibinfo{author}{A.~Jentzen},
\newblock \bibinfo{title}{Solving stochastic differential equations and
  {K}olmogorov equations by means of deep learning},
\newblock \bibinfo{journal}{arXiv preprint arXiv:1806.00421}
  (\bibinfo{year}{2018}).
%Type = Article
\bibitem[{Beck et~al.(2019{\natexlab{a}})Beck, Becker, Cheridito, Jentzen, and
  Neufeld}]{beck2019deep}
\bibinfo{author}{C.~Beck}, \bibinfo{author}{S.~Becker},
  \bibinfo{author}{P.~Cheridito}, \bibinfo{author}{A.~Jentzen},
  \bibinfo{author}{A.~Neufeld},
\newblock \bibinfo{title}{Deep splitting method for parabolic {PDE}s},
\newblock \bibinfo{journal}{arXiv preprint arXiv:1907.03452}
  (\bibinfo{year}{2019}{\natexlab{a}}).
%Type = Article
\bibitem[{Beck et~al.(2019{\natexlab{b}})Beck, Weinan, and
  Jentzen}]{beck2019machine}
\bibinfo{author}{C.~Beck}, \bibinfo{author}{E.~Weinan},
  \bibinfo{author}{A.~Jentzen},
\newblock \bibinfo{title}{Machine learning approximation algorithms for
  high-dimensional fully nonlinear partial differential equations and
  second-order backward stochastic differential equations},
\newblock \bibinfo{journal}{Journal of Nonlinear Science} \bibinfo{volume}{29}
  (\bibinfo{year}{2019}{\natexlab{b}}) \bibinfo{pages}{1563--1619}.
%Type = Article
\bibitem[{Berg and Nyström(2018)}]{Berg2018Unified}
\bibinfo{author}{J.~Berg}, \bibinfo{author}{K.~Nyström},
\newblock \bibinfo{title}{A unified deep artificial neural network approach to
  partial differential equations in complex geometries},
\newblock \bibinfo{journal}{Neurocomputing} \bibinfo{volume}{317}
  (\bibinfo{year}{2018}) \bibinfo{pages}{28 -- 41}.
%Type = Article
\bibitem[{Chan-Wai-Nam et~al.(2019)Chan-Wai-Nam, Mikael, and
  Warin}]{chan2019machine}
\bibinfo{author}{Q.~Chan-Wai-Nam}, \bibinfo{author}{J.~Mikael},
  \bibinfo{author}{X.~Warin},
\newblock \bibinfo{title}{Machine learning for semi linear {PDE}s},
\newblock \bibinfo{journal}{Journal of Scientific Computing}
  \bibinfo{volume}{79} (\bibinfo{year}{2019}) \bibinfo{pages}{1667--1712}.
%Type = Inproceedings
\bibitem[{{Cheng} and {Lewis}(2006)}]{Cheng2006Fixed}
\bibinfo{author}{T.~{Cheng}}, \bibinfo{author}{F.~L. {Lewis}},
\newblock \bibinfo{title}{Fixed-final time constrained optimal control of
  nonlinear systems using neural network {HJB} approach},
\newblock in: \bibinfo{booktitle}{Proceedings of the 45th IEEE Conference on
  Decision and Control}, \bibinfo{year}{2006}, pp. \bibinfo{pages}{3016--3021}.
  \DOIprefix\doi{10.1109/CDC.2006.377523}.
%Type = Article
\bibitem[{Dissanayake and Phan-Thien(1994)}]{Dissanayake1994Neural}
\bibinfo{author}{M.~W. M.~G. Dissanayake}, \bibinfo{author}{N.~Phan-Thien},
\newblock \bibinfo{title}{Neural-network-based approximations for solving
  partial differential equations},
\newblock \bibinfo{journal}{Communications in Numerical Methods in Engineering}
  \bibinfo{volume}{10} (\bibinfo{year}{1994}) \bibinfo{pages}{195--201}.
%Type = Article
\bibitem[{Dockhorn(2019)}]{dockhorn2019discussion}
\bibinfo{author}{T.~Dockhorn},
\newblock \bibinfo{title}{A discussion on solving partial differential
  equations using neural networks},
\newblock \bibinfo{journal}{arXiv preprint arXiv:1904.07200}
  (\bibinfo{year}{2019}).
%Type = Article
\bibitem[{E et~al.(2017)E, Han, and Jentzen}]{E2017Deep}
\bibinfo{author}{W.~E}, \bibinfo{author}{J.~Han}, \bibinfo{author}{A.~Jentzen},
\newblock \bibinfo{title}{Deep learning-based numerical methods for
  high-dimensional parabolic partial differential equations and backward
  stochastic differential equations},
\newblock \bibinfo{journal}{Communications in Mathematics and Statistics}
  \bibinfo{volume}{5} (\bibinfo{year}{2017}) \bibinfo{pages}{349--380}.
%Type = Article
\bibitem[{{Farimani} et~al.(2017){Farimani}, {Gomes}, and
  {Pande}}]{Farimani2017Deep}
\bibinfo{author}{A.~B. {Farimani}}, \bibinfo{author}{J.~{Gomes}},
  \bibinfo{author}{V.~S. {Pande}},
\newblock \bibinfo{title}{{Deep Learning the Physics of Transport Phenomena}},
\newblock \bibinfo{journal}{arXiv e-prints}  (\bibinfo{year}{2017}).
%Type = Article
\bibitem[{Fujii et~al.(2019)Fujii, Takahashi, and
  Takahashi}]{Fujii2019Asymptotic}
\bibinfo{author}{M.~Fujii}, \bibinfo{author}{A.~Takahashi},
  \bibinfo{author}{M.~Takahashi},
\newblock \bibinfo{title}{Asymptotic expansion as prior knowledge in deep
  learning method for high dimensional {BSDE}s},
\newblock \bibinfo{journal}{Asia-Pacific Financial Markets}
  \bibinfo{volume}{26} (\bibinfo{year}{2019}) \bibinfo{pages}{391--408}.
%Type = Article
\bibitem[{Grohs et~al.(2019)Grohs, Jentzen, and Salimova}]{grohs2019deep}
\bibinfo{author}{P.~Grohs}, \bibinfo{author}{A.~Jentzen},
  \bibinfo{author}{D.~Salimova},
\newblock \bibinfo{title}{Deep neural network approximations for {M}onte
  {C}arlo algorithms},
\newblock \bibinfo{journal}{arXiv preprint arXiv:1908.10828}
  (\bibinfo{year}{2019}).
%Type = Article
\bibitem[{Han et~al.(2019)Han, Zhang, and Weinan}]{han2019solving}
\bibinfo{author}{J.~Han}, \bibinfo{author}{L.~Zhang},
  \bibinfo{author}{E.~Weinan},
\newblock \bibinfo{title}{Solving many-electron {S}chr{\"o}dinger equation
  using deep neural networks},
\newblock \bibinfo{journal}{Journal of Computational Physics}
  (\bibinfo{year}{2019}) \bibinfo{pages}{108929}.
%Type = Inproceedings
\bibitem[{Hsieh et~al.(2019)Hsieh, Zhao, Eismann, Mirabella, and
  Ermon}]{hsieh2018learning}
\bibinfo{author}{J.-T. Hsieh}, \bibinfo{author}{S.~Zhao},
  \bibinfo{author}{S.~Eismann}, \bibinfo{author}{L.~Mirabella},
  \bibinfo{author}{S.~Ermon},
\newblock \bibinfo{title}{Learning neural {PDE} solvers with convergence
  guarantees},
\newblock in: \bibinfo{booktitle}{International Conference on Learning
  Representations}, \bibinfo{year}{2019}.
%Type = Article
\bibitem[{Jianyu et~al.(2003)Jianyu, Siwei, Yingjian, and
  Yaping}]{jianyu2003numerical}
\bibinfo{author}{L.~Jianyu}, \bibinfo{author}{L.~Siwei},
  \bibinfo{author}{Q.~Yingjian}, \bibinfo{author}{H.~Yaping},
\newblock \bibinfo{title}{Numerical solution of elliptic partial differential
  equation using radial basis function neural networks},
\newblock \bibinfo{journal}{Neural Networks} \bibinfo{volume}{16}
  (\bibinfo{year}{2003}) \bibinfo{pages}{729--734}.
%Type = Article
\bibitem[{Khoo et~al.(2017)Khoo, Lu, and Ying}]{khoo2017solving}
\bibinfo{author}{Y.~Khoo}, \bibinfo{author}{J.~Lu}, \bibinfo{author}{L.~Ying},
\newblock \bibinfo{title}{Solving parametric {PDE} problems with artificial
  neural networks},
\newblock \bibinfo{journal}{arXiv preprint arXiv:1707.03351}
  (\bibinfo{year}{2017}).
%Type = Article
\bibitem[{Khoo et~al.(2019)Khoo, Lu, and Ying}]{khoo2019solving}
\bibinfo{author}{Y.~Khoo}, \bibinfo{author}{J.~Lu}, \bibinfo{author}{L.~Ying},
\newblock \bibinfo{title}{Solving for high-dimensional committor functions
  using artificial neural networks},
\newblock \bibinfo{journal}{Research in the Mathematical Sciences}
  \bibinfo{volume}{6} (\bibinfo{year}{2019}) \bibinfo{pages}{1}.
%Type = Article
\bibitem[{{Lagaris} et~al.(1998){Lagaris}, {Likas}, and
  {Fotiadis}}]{Lagaris1998ANN}
\bibinfo{author}{I.~E. {Lagaris}}, \bibinfo{author}{A.~{Likas}},
  \bibinfo{author}{D.~I. {Fotiadis}},
\newblock \bibinfo{title}{Artificial neural networks for solving ordinary and
  partial differential equations},
\newblock \bibinfo{journal}{IEEE Transactions on Neural Networks}
  \bibinfo{volume}{9} (\bibinfo{year}{1998}) \bibinfo{pages}{987--1000}.
%Type = Article
\bibitem[{{Lagaris} et~al.(2000){Lagaris}, {Likas}, and
  {Papageorgiou}}]{Lagaris2000NN}
\bibinfo{author}{I.~E. {Lagaris}}, \bibinfo{author}{A.~C. {Likas}},
  \bibinfo{author}{D.~G. {Papageorgiou}},
\newblock \bibinfo{title}{Neural-network methods for boundary value problems
  with irregular boundaries},
\newblock \bibinfo{journal}{IEEE Transactions on Neural Networks}
  \bibinfo{volume}{11} (\bibinfo{year}{2000}) \bibinfo{pages}{1041--1049}.
%Type = Article
\bibitem[{Lee and Kang(1990)}]{lee1990neural}
\bibinfo{author}{H.~Lee}, \bibinfo{author}{I.~S. Kang},
\newblock \bibinfo{title}{Neural algorithm for solving differential equations},
\newblock \bibinfo{journal}{Journal of Computational Physics}
  \bibinfo{volume}{91} (\bibinfo{year}{1990}) \bibinfo{pages}{110--131}.
%Type = Article
\bibitem[{Lye et~al.(2019)Lye, Mishra, and Ray}]{lye2019deep}
\bibinfo{author}{K.~O. Lye}, \bibinfo{author}{S.~Mishra},
  \bibinfo{author}{D.~Ray},
\newblock \bibinfo{title}{Deep learning observables in computational fluid
  dynamics},
\newblock \bibinfo{journal}{arXiv preprint arXiv:1903.03040}
  (\bibinfo{year}{2019}).
%Type = Article
\bibitem[{{McFall} and {Mahan}(2009)}]{McFall2009ANN}
\bibinfo{author}{K.~S. {McFall}}, \bibinfo{author}{J.~R. {Mahan}},
\newblock \bibinfo{title}{Artificial neural network method for solution of
  boundary value problems with exact satisfaction of arbitrary boundary
  conditions},
\newblock \bibinfo{journal}{IEEE Transactions on Neural Networks}
  \bibinfo{volume}{20} (\bibinfo{year}{2009}) \bibinfo{pages}{1221--1233}.
%Type = Article
\bibitem[{Meade and Fernandez(1994)}]{Meade1994Numerical}
\bibinfo{author}{A.~Meade}, \bibinfo{author}{A.~Fernandez},
\newblock \bibinfo{title}{The numerical solution of linear ordinary
  differential equations by feedforward neural networks},
\newblock \bibinfo{journal}{Mathematical and Computer Modelling}
  \bibinfo{volume}{19} (\bibinfo{year}{1994}) \bibinfo{pages}{1 -- 25}.
%Type = Article
\bibitem[{van Milligen et~al.(1995)van Milligen, Tribaldos, and
  Jim\'enez}]{Milligen1995NN}
\bibinfo{author}{B.~P. van Milligen}, \bibinfo{author}{V.~Tribaldos},
  \bibinfo{author}{J.~A. Jim\'enez},
\newblock \bibinfo{title}{Neural network differential equation and plasma
  equilibrium solver},
\newblock \bibinfo{journal}{Phys. Rev. Lett.} \bibinfo{volume}{75}
  (\bibinfo{year}{1995}) \bibinfo{pages}{3594--3597}.
%Type = Article
\bibitem[{Pham et~al.(2019)Pham, Pham, and Warin}]{pham2019neural}
\bibinfo{author}{H.~Pham}, \bibinfo{author}{H.~Pham},
  \bibinfo{author}{X.~Warin},
\newblock \bibinfo{title}{Neural networks-based backward scheme for fully
  nonlinear {PDE}s},
\newblock \bibinfo{journal}{arXiv preprint arXiv:1908.00412}
  (\bibinfo{year}{2019}).
%Type = Article
\bibitem[{{Rudd} et~al.(2014){Rudd}, {Muro}, and
  {Ferrari}}]{Rudd2014Constrained}
\bibinfo{author}{K.~{Rudd}}, \bibinfo{author}{G.~D. {Muro}},
  \bibinfo{author}{S.~{Ferrari}},
\newblock \bibinfo{title}{A constrained backpropagation approach for the
  adaptive solution of partial differential equations},
\newblock \bibinfo{journal}{IEEE Transactions on Neural Networks and Learning
  Systems} \bibinfo{volume}{25} (\bibinfo{year}{2014})
  \bibinfo{pages}{571--584}.
%Type = Inproceedings
\bibitem[{{Tang} et~al.(2017){Tang}, {Shan}, {Dang}, {Li}, {Yang}, {Xu}, and
  {Wu}}]{Tang2017Study}
\bibinfo{author}{W.~{Tang}}, \bibinfo{author}{T.~{Shan}},
  \bibinfo{author}{X.~{Dang}}, \bibinfo{author}{M.~{Li}},
  \bibinfo{author}{F.~{Yang}}, \bibinfo{author}{S.~{Xu}},
  \bibinfo{author}{J.~{Wu}},
\newblock \bibinfo{title}{Study on a {P}oisson's equation solver based on deep
  learning technique},
\newblock in: \bibinfo{booktitle}{2017 IEEE Electrical Design of Advanced
  Packaging and Systems Symposium (EDAPS)}, \bibinfo{year}{2017}, pp.
  \bibinfo{pages}{1--3}. \DOIprefix\doi{10.1109/EDAPS.2017.8277017}.
%Type = Article
\bibitem[{{Tassa} and {Erez}(2007)}]{Tassa2007Least}
\bibinfo{author}{Y.~{Tassa}}, \bibinfo{author}{T.~{Erez}},
\newblock \bibinfo{title}{Least squares solutions of the {HJB} equation with
  neural network value-function approximators},
\newblock \bibinfo{journal}{IEEE Transactions on Neural Networks}
  \bibinfo{volume}{18} (\bibinfo{year}{2007}) \bibinfo{pages}{1031--1041}.
%Type = Article
\bibitem[{Weinan and Yu(2018)}]{weinan2018deep}
\bibinfo{author}{E.~Weinan}, \bibinfo{author}{B.~Yu},
\newblock \bibinfo{title}{The deep {R}itz method: a deep learning-based
  numerical algorithm for solving variational problems},
\newblock \bibinfo{journal}{Communications in Mathematics and Statistics}
  \bibinfo{volume}{6} (\bibinfo{year}{2018}) \bibinfo{pages}{1--12}.
%Type = Book
\bibitem[{Yadav et~al.(2015)Yadav, Yadav, and Kumar}]{Yadav2015Intro}
\bibinfo{author}{N.~Yadav}, \bibinfo{author}{A.~Yadav},
  \bibinfo{author}{M.~Kumar}, \bibinfo{title}{An introduction to neural network
  methods for differential equations}, SpringerBriefs in Applied Sciences and
  Technology, \bibinfo{publisher}{Springer, Dordrecht}, \bibinfo{year}{2015}.
  \DOIprefix\doi{10.1007/978-94-017-9816-7}.
%Type = Article
\bibitem[{Yang et~al.(2018)Yang, Zhang, and Karniadakis}]{yang2018physics}
\bibinfo{author}{L.~Yang}, \bibinfo{author}{D.~Zhang}, \bibinfo{author}{G.~E.
  Karniadakis},
\newblock \bibinfo{title}{Physics-informed generative adversarial networks for
  stochastic differential equations},
\newblock \bibinfo{journal}{arXiv preprint arXiv:1811.02033}
  (\bibinfo{year}{2018}).
%Type = Article
\bibitem[{Yang and Perdikaris(2019)}]{yang2019adversarial}
\bibinfo{author}{Y.~Yang}, \bibinfo{author}{P.~Perdikaris},
\newblock \bibinfo{title}{Adversarial uncertainty quantification in
  physics-informed neural networks},
\newblock \bibinfo{journal}{Journal of Computational Physics}
  \bibinfo{volume}{394} (\bibinfo{year}{2019}) \bibinfo{pages}{136--152}.
%Type = Article
\bibitem[{Zhao et~al.(2017)Zhao, Zhou, and Kong}]{Zhao2017High}
\bibinfo{author}{W.~Zhao}, \bibinfo{author}{T.~Zhou},
  \bibinfo{author}{T.~Kong},
\newblock \bibinfo{title}{High order numerical schemes for second-order
  {FBSDE}s with applications to stochastic optimal control},
\newblock \bibinfo{journal}{Commun. Comput. Phys.} \bibinfo{volume}{21}
  (\bibinfo{year}{2017}) \bibinfo{pages}{808--834}.
%Type = Article
\bibitem[{Kong et~al.(2015)Kong, Zhao, and Zhou}]{Kong2015Probabilistic}
\bibinfo{author}{T.~Kong}, \bibinfo{author}{W.~Zhao},
  \bibinfo{author}{T.~Zhou},
\newblock \bibinfo{title}{Probabilistic high order numerical schemes for fully
  nonlinear parabolic {PDE}s},
\newblock \bibinfo{journal}{Commun. Comput. Phys.} \bibinfo{volume}{18}
  (\bibinfo{year}{2015}) \bibinfo{pages}{1482--1503}.
%Type = Article
\bibitem[{Long et~al.(2017)Long, Lu, Ma, and Dong}]{long2017pde}
\bibinfo{author}{Z.~Long}, \bibinfo{author}{Y.~Lu}, \bibinfo{author}{X.~Ma},
  \bibinfo{author}{B.~Dong},
\newblock \bibinfo{title}{{PDE}-net: Learning {PDE}s from data},
\newblock \bibinfo{journal}{arXiv preprint arXiv:1710.09668}
  (\bibinfo{year}{2017}).
%Type = Article
\bibitem[{Long et~al.(2019)Long, Lu, and Dong}]{long2019pde}
\bibinfo{author}{Z.~Long}, \bibinfo{author}{Y.~Lu}, \bibinfo{author}{B.~Dong},
\newblock \bibinfo{title}{{PDE}-net 2.0: Learning {PDE}s from data with a
  numeric-symbolic hybrid deep network},
\newblock \bibinfo{journal}{Journal of Computational Physics}
  \bibinfo{volume}{399} (\bibinfo{year}{2019}) \bibinfo{pages}{108925}.
%Type = Article
\bibitem[{Meng and Karniadakis(2019)}]{meng2019composite}
\bibinfo{author}{X.~Meng}, \bibinfo{author}{G.~E. Karniadakis},
\newblock \bibinfo{title}{A composite neural network that learns from
  multi-fidelity data: Application to function approximation and inverse {PDE}
  problems},
\newblock \bibinfo{journal}{arXiv preprint arXiv:1903.00104}
  (\bibinfo{year}{2019}).
%Type = Article
\bibitem[{Meng et~al.(2019)Meng, Li, Zhang, and Karniadakis}]{meng2019ppinn}
\bibinfo{author}{X.~Meng}, \bibinfo{author}{Z.~Li}, \bibinfo{author}{D.~Zhang},
  \bibinfo{author}{G.~E. Karniadakis},
\newblock \bibinfo{title}{{PPINN}: Parareal physics-informed neural network for
  time-dependent {PDE}s},
\newblock \bibinfo{journal}{arXiv preprint arXiv:1909.10145}
  (\bibinfo{year}{2019}).
%Type = Article
\bibitem[{Pang et~al.(2019)Pang, Lu, and Karniadakis}]{pang2019fpinns}
\bibinfo{author}{G.~Pang}, \bibinfo{author}{L.~Lu}, \bibinfo{author}{G.~E.
  Karniadakis},
\newblock \bibinfo{title}{f{PINN}s: Fractional physics-informed neural
  networks},
\newblock \bibinfo{journal}{SIAM Journal on Scientific Computing}
  \bibinfo{volume}{41} (\bibinfo{year}{2019}) \bibinfo{pages}{A2603--A2626}.
%Type = Article
\bibitem[{Raissi(2018{\natexlab{a}})}]{raissi2018deep}
\bibinfo{author}{M.~Raissi},
\newblock \bibinfo{title}{Deep hidden physics models: Deep learning of
  nonlinear partial differential equations},
\newblock \bibinfo{journal}{The Journal of Machine Learning Research}
  \bibinfo{volume}{19} (\bibinfo{year}{2018}{\natexlab{a}})
  \bibinfo{pages}{932--955}.
%Type = Article
\bibitem[{Raissi(2018{\natexlab{b}})}]{raissi2018forward}
\bibinfo{author}{M.~Raissi},
\newblock \bibinfo{title}{Forward-backward stochastic neural networks: Deep
  learning of high-dimensional partial differential equations},
\newblock \bibinfo{journal}{arXiv preprint arXiv:1804.07010}
  (\bibinfo{year}{2018}{\natexlab{b}}).
%Type = Article
\bibitem[{Raissi et~al.(2017{\natexlab{a}})Raissi, Perdikaris, and
  Karniadakis}]{raissi2017physicsi}
\bibinfo{author}{M.~Raissi}, \bibinfo{author}{P.~Perdikaris},
  \bibinfo{author}{G.~E. Karniadakis},
\newblock \bibinfo{title}{Physics informed deep learning (part i): Data-driven
  solutions of nonlinear partial differential equations},
\newblock \bibinfo{journal}{arXiv preprint arXiv:1711.10561}
  (\bibinfo{year}{2017}{\natexlab{a}}).
%Type = Article
\bibitem[{Raissi et~al.(2017{\natexlab{b}})Raissi, Perdikaris, and
  Karniadakis}]{raissi2017physicsii}
\bibinfo{author}{M.~Raissi}, \bibinfo{author}{P.~Perdikaris},
  \bibinfo{author}{G.~E. Karniadakis},
\newblock \bibinfo{title}{Physics informed deep learning (part ii): Data-driven
  discovery of nonlinear partial differential equations},
\newblock \bibinfo{journal}{arXiv preprint arXiv:1711.10566}
  (\bibinfo{year}{2017}{\natexlab{b}}).
%Type = Article
\bibitem[{Raissi et~al.(2019)Raissi, Perdikaris, and
  Karniadakis}]{Raissi2019PINN}
\bibinfo{author}{M.~Raissi}, \bibinfo{author}{P.~Perdikaris},
  \bibinfo{author}{G.~Karniadakis},
\newblock \bibinfo{title}{Physics-informed neural networks: A deep learning
  framework for solving forward and inverse problems involving nonlinear
  partial differential equations},
\newblock \bibinfo{journal}{Journal of Computational Physics}
  \bibinfo{volume}{378} (\bibinfo{year}{2019}) \bibinfo{pages}{686 -- 707}.
%Type = Inproceedings
\bibitem[{Uchiyama and Sonehara(1993)}]{uchiyama1993solving}
\bibinfo{author}{T.~Uchiyama}, \bibinfo{author}{N.~Sonehara},
\newblock \bibinfo{title}{Solving inverse problems in nonlinear {PDE}s by
  recurrent neural networks},
\newblock in: \bibinfo{booktitle}{IEEE International Conference on Neural
  Networks}, \bibinfo{organization}{IEEE}, \bibinfo{year}{1993}, pp.
  \bibinfo{pages}{99--102}.
%Type = Article
\bibitem[{Zhang et~al.(2019{\natexlab{a}})Zhang, Guo, and
  Karniadakis}]{zhang2019learning}
\bibinfo{author}{D.~Zhang}, \bibinfo{author}{L.~Guo}, \bibinfo{author}{G.~E.
  Karniadakis},
\newblock \bibinfo{title}{Learning in modal space: Solving time-dependent
  stochastic {PDE}s using physics-informed neural networks},
\newblock \bibinfo{journal}{arXiv preprint arXiv:1905.01205}
  (\bibinfo{year}{2019}{\natexlab{a}}).
%Type = Article
\bibitem[{Zhang et~al.(2019{\natexlab{b}})Zhang, Lu, Guo, and
  Karniadakis}]{zhang2019quantifying}
\bibinfo{author}{D.~Zhang}, \bibinfo{author}{L.~Lu}, \bibinfo{author}{L.~Guo},
  \bibinfo{author}{G.~E. Karniadakis},
\newblock \bibinfo{title}{Quantifying total uncertainty in physics-informed
  neural networks for solving forward and inverse stochastic problems},
\newblock \bibinfo{journal}{Journal of Computational Physics}
  \bibinfo{volume}{397} (\bibinfo{year}{2019}{\natexlab{b}})
  \bibinfo{pages}{108850}.
%Type = Article
\bibitem[{Fan and Ying(2019)}]{fan2019solving}
\bibinfo{author}{Y.~Fan}, \bibinfo{author}{L.~Ying},
\newblock \bibinfo{title}{Solving inverse wave scattering with deep learning},
\newblock \bibinfo{journal}{arXiv preprint arXiv:1911.13202}
  (\bibinfo{year}{2019}).
%Type = Article
\bibitem[{Yan and Zhou(2019)}]{yan2019adaptive}
\bibinfo{author}{L.~Yan}, \bibinfo{author}{T.~Zhou},
\newblock \bibinfo{title}{An adaptive surrogate modeling based on deep neural
  networks for large-scale {B}ayesian inverse problems},
\newblock \bibinfo{journal}{arXiv preprint arXiv:1911.08926}
  (\bibinfo{year}{2019}).
%Type = Article
\bibitem[{C{\'a}rdenas and Gibou(2020)}]{cardenas2020deep}
\bibinfo{author}{L.~{\'A}.~L. C{\'a}rdenas}, \bibinfo{author}{F.~Gibou},
\newblock \bibinfo{title}{A deep learning approach for the computation of
  curvature in the level-set method},
\newblock \bibinfo{journal}{arXiv preprint arXiv:2002.02804}
  (\bibinfo{year}{2020}).
%Type = Book
\bibitem[{Hiriart-Urruty and
  Lemar{\'e}chal(1993{\natexlab{a}})}]{hiriart2013convexI}
\bibinfo{author}{J.-B. Hiriart-Urruty}, \bibinfo{author}{C.~Lemar{\'e}chal},
  \bibinfo{title}{Convex analysis and minimization algorithms I: Fundamentals},
  volume \bibinfo{volume}{305}, \bibinfo{publisher}{Springer science \&
  business media}, \bibinfo{year}{1993}{\natexlab{a}}.
%Type = Book
\bibitem[{Hiriart-Urruty and
  Lemar{\'e}chal(1993{\natexlab{b}})}]{hiriart2013convexII}
\bibinfo{author}{J.-B. Hiriart-Urruty}, \bibinfo{author}{C.~Lemar{\'e}chal},
  \bibinfo{title}{Convex analysis and minimization algorithms II: Advanced
  Theory and Bundle Methods}, volume \bibinfo{volume}{306},
  \bibinfo{publisher}{Springer science \& business media},
  \bibinfo{year}{1993}{\natexlab{b}}.
%Type = Book
\bibitem[{Rockafellar(1970)}]{rockafellar1970convex}
\bibinfo{author}{R.~T. Rockafellar}, \bibinfo{title}{Convex analysis},
  \bibinfo{publisher}{Princeton university press}, \bibinfo{year}{1970}.
%Type = Book
\bibitem[{Evans(2010)}]{evans1998partial}
\bibinfo{author}{L.~C. Evans}, \bibinfo{title}{Partial differential equations},
  volume~\bibinfo{volume}{19} of \textit{\bibinfo{series}{Graduate Studies in
  Mathematics}}, \bibinfo{edition}{second} ed., \bibinfo{publisher}{American
  Mathematical Society, Providence, RI}, \bibinfo{year}{2010}.
  \DOIprefix\doi{10.1090/gsm/019}.
%Type = Article
\bibitem[{Hopf(1965)}]{Hopf1965}
\bibinfo{author}{E.~Hopf},
\newblock \bibinfo{title}{Generalized solutions of non-linear equations of
  first order},
\newblock \bibinfo{journal}{J. Math. Mech.} \bibinfo{volume}{14}
  (\bibinfo{year}{1965}) \bibinfo{pages}{951--973}.
%Type = Article
\bibitem[{Dragoni(2007)}]{Dragoni2007Metric}
\bibinfo{author}{F.~Dragoni},
\newblock \bibinfo{title}{Metric {H}opf-{L}ax formula with semicontinuous
  data},
\newblock \bibinfo{journal}{Discrete Contin. Dyn. Syst.} \bibinfo{volume}{17}
  (\bibinfo{year}{2007}) \bibinfo{pages}{713--729}.

\end{thebibliography}

\end{document}